\documentclass[12pt,a4paper,reqno]{amsart}
\usepackage{amsmath}
\usepackage{amsfonts}
\usepackage{amssymb}
\usepackage{amsthm}
\usepackage{enumerate}
\usepackage{centernot}
\usepackage{fullpage}
\usepackage{mathrsfs}
\usepackage{graphicx}

\newcommand{\upperRomannumeral}[1]{\uppercase\expandafter{\romannumeral#1}}

\theoremstyle{plain}

  \newtheorem{proposition}[]{Proposition}
  \newtheorem{lemma}[]{Lemma}
  \newtheorem{theorem}[]{Theorem}
  \newtheorem{corollary}[]{Corollary}
  
  \newtheorem{remark}[]{Remark}
  \newtheorem*{remark*}{Remark}

\title{FK-Ising coupling applied to near-critical planar models}
\author{Federico Camia}
\address{Division of Science, NYU Abu Dhabi, Saadiyat Island, Abu Dhabi, UAE \& Department of Mathematics, VU Amsterdam, De Boelelaan 1081a, 1081 HV Amsterdam, the Netherlands}
\email{federico.camia@nyu.edu}
\author{Jianping Jiang}
\address{NYU-ECNU Institute of Mathematical Sciences at NYU Shanghai, 3663 Zhongshan
Road North, Shanghai 200062, China.}
\email{jjiang@nyu.edu}
\author{Charles M. Newman}
\address{Courant Institute of Mathematical Sciences, New York University,
251 Mercer st, New York, NY 10012, USA, \& NYU-ECNU Institute of Mathematical
Sciences at NYU Shanghai, 3663 Zhongshan Road North, Shanghai 200062, China.}
\email{newman@cims.nyu.edu}
\begin{document}
\begin{abstract}
We consider the Ising model at its critical temperature with external magnetic field $ha^{15/8}$ on $a\mathbb{Z}^2$. We give a purely probabilistic proof, using FK methods rather than reflection positivity, that for $a=1$, the correlation length is $\geq const.~h^{-8/15}$ as $h\downarrow0$. We extend to the $a\downarrow0$ continuum limit the FK-Ising coupling for all $h>0$, and obtain tail estimates for the largest renormalized cluster area in a finite domain as well as an upper bound with exponent $1/8$ for the one-arm event. Finally, we show that for $a=1$, the average magnetization, $\mathcal{M}(h)$, in $\mathbb{Z}^2$ satisfies $\mathcal{M}(h)/h^{1/15}\rightarrow$ some $B\in(0,\infty)$ as $h\downarrow0$.
\end{abstract}
\maketitle
\section{Introduction}
\subsection{Overview}\label{secoverview}
In a recent paper \cite{CJN17}, the authors obtained upper and lower bounds of the form $C_0H^{8/15}$ and $B_0H^{8/15}$ as $H\downarrow0$, for the exponential decay rate (the mass or inverse correlation length) of the $(\beta_c,H)$ planar ($\mathbb{Z}^2$) Ising model at critical inverse temperature $\beta_c$ with magnetic field $H\geq 0$. The lower bound derivation used methods based on the FK random cluster representation of the Ising model, including the use of the Radon-Nikodym derivative of the distribution of the model with external field with respect to the distribution of the model without external field. This derivative appears implicitly in \cite{CV16} (see also [ref. 16 of \cite{CV16}]), as we learned after the first version of this paper was posted in 2017. The upper bound, on the other hand, was derived in \cite{CJN17} by quite different methods based on reflection positivity.

Here we extend the FK methods of \cite{CJN17} in several ways. First we give (in Theorem \ref{thm:upper} and Corollary \ref{cor:upper}) an alternative derivation of the $H^{8/15}$ upper bound using only FK-based methods. Then in Theorem \ref{thm2} we show that the FK/Ising coupling of \cite{CJN17} is valid for the scaling limit continuum FK measure ensemble with positive renormalized magnetic field $h$, extending the continuum Edwards-Sokal type coupling shown in \cite{CCK17} beyond the $h=0$ case. This coupling is then applied to obtain in Theorem \ref{thm3} for $h\geq 0$ precise tail behavior (of the form $\exp{(-Cx^{16})}$) for the largest total mass in the ensemble of continuum FK measures in a bounded domain; this is analogous to the result of \cite{Kis14} for the tail of the largest cluster area in critical Bernoulli percolation. Tail behavior for both continuum and discrete FK models is derived in Sections \ref{sec:coupling} and \ref{sec:proof} by using the FK/Ising coupling to relate moment generating functions for cluster size to those for Ising magnetization.

Our final main result (in Theorem \ref{thm4}) gives very precise behavior for the magnetization $\mathcal{M}(H)$ (expected spin value of the $(\beta_c,H)$ Ising model on $\mathbb{Z}^2$) that improves the bounds from \cite{CGN14} that as $H\downarrow 0$
\begin{equation}\label{eq:mbd}
B_1 H^{1/15}\leq \mathcal{M}(H)\leq B_2 H^{1/15}.
\end{equation}
The improved result is
\begin{equation}\label{eq:mbd1}
\lim_{H\downarrow 0}\frac{\mathcal{M}(H)}{H^{1/15}}=B\in (0,\infty).
\end{equation}
The derivation of \eqref{eq:mbd} in \cite{CGN14} was fairly short, but the derivation of \eqref{eq:mbd1} in Subsection~\ref{subsec:main} below is yet shorter and uses little more than the existence of a scaling limit magnetization field for $h\geq 0$ \cite{CGN15,CGN16}.

\subsection{Main results}\label{subsec:main}

Let $a>0$. Denote by $P_h^a$ the infinite volume Ising measure at the inverse critical temperature $\beta_c$ on $a\mathbb{Z}^2$ with external field $a^{15/8}h>0$. Let $\langle\cdot\rangle_{a,h}$ be the expectation with respect to $P_h^a$. Let $\langle \sigma_x;\sigma_y\rangle_{a,h}$ be the truncated two-point function, i.e.,
\[\langle \sigma_x;\sigma_y\rangle_{a,h}:=\langle \sigma_x\sigma_y\rangle_{a,h}-\langle\sigma_x\rangle_{a,h}\langle\sigma_y\rangle_{a,h}.\]
For $x,y\in\mathbb{R}^2$, let $|x-y|:=\|x-y\|_2$ denote the Euclidean distance. Our first main result is:
\begin{theorem}\label{thm:upper}
There exist $C_2, C_3, C_4\in(0,\infty)$ such that for any $a\in(0,1]$, $h>0$ with $a^{15/8}h\leq 1$, and $x,y\in a\mathbb{Z}^2$ with $|x-y|\geq C_2h^{-8/15}$
\begin{equation}\label{eq:up1}
\langle \sigma_x;\sigma_y\rangle_{a,h}\geq C_3a^{1/4}h^{2/15}e^{-C_4 h^{8/15}|x-y|}.
\end{equation}
In particular, for a=1 and any $H\in(0,1]$, we have for any $x^{\prime},y^{\prime}\in \mathbb{Z}^2$  with $|x^{\prime}-y^{\prime}|\geq C_2H^{-8/15}$
\begin{equation}\label{eq:up2}
\langle \sigma_{x^{\prime}};\sigma_{y^{\prime}}\rangle_{1,H}\geq C_3 H^{2/15}e^{-C_4 H^{8/15}|x^{\prime}-y^{\prime}|}.
\end{equation}
\end{theorem}

For $a=1$, define the (lattice) mass (or inverse correlation length) $\tilde{M}(H)$ as the supremum of all $\tilde{m}>0$ such that for some $C(\tilde{m})<\infty$,
\begin{equation}\label{eq:def}
\langle \sigma_{x^{\prime}};\sigma_{y^{\prime}}\rangle_{1,H}\leq C(\tilde{m})e^{-\tilde{m}|x^{\prime}-y^{\prime}|}\text{ for any }x^{\prime},y^{\prime}\in \mathbb{Z}^2.
\end{equation}
The following immediate corollary of Theorem \ref{thm:upper} gives a one-sided bound for the behavior of $\tilde{M}(H)$ as $H\downarrow 0$, with the expected critical exponent $8/15$.
\begin{corollary}\label{cor:upper}
\[\tilde{M}(H)\leq C_4 H^{8/15} \text{ as }H\downarrow0,\]
with $C_4$ the same constant as in Theorem \ref{thm:upper}.
\end{corollary}

Let $D\subseteq\mathbb{R}^2$ be a simply-connected and bounded domain with a piecewise smooth boundary. Let $\Phi^{a,h}_D$ be the near-critical magnetization field in $D^a:=a\mathbb{Z}^2\cap D$ defined by
\begin{equation}\label{eqfield}
\Phi^{a,h}_D:=a^{15/8}\sum_{x\in D^a}\sigma_x\delta_x,
\end{equation}
where $\{\sigma_x\}_{x\in D^a}$ is a configuration for the critical Ising model on $D^a$ with external field $a^{15/8}h$ and free boundary conditions and $\delta_x$ is a unit Dirac point measure at $x$. In Proposition 1.5 of \cite{CGN16} (resp., Theorem 1.3 of \cite{CGN15}), it was proved that $\Phi^{a,h}_{D}$ (resp., $\Phi^{a,0}_{D}$) converges in law to a continuum (generalized) random field $\Phi^{h}_D$ (resp., $\Phi^{0}_D$). Let $C^{\infty}(D)$ denote the set of infinitely differentiable functions with domain $D$. $\Phi^{h}_D(\tilde{f})$ denotes the field $\Phi^{h}_D$ paired against the test function $\tilde{f}$ (which was denoted $\langle\Phi^{h}_D,\tilde{f}\rangle$ in \cite{CGN16}). For any configuration $\omega$ in the FK percolation on $D^a$ with no external field and free boundary conditions, let $\mathscr{C}(D^a,f,\omega)$ denote the set of clusters of $\omega$ in $D^a$, where $f$ stands for free boundary conditions. For $\mathcal{C}\in\mathscr{C}(D^a,f,\cdot)$, let $\mu^a_{\mathcal{C}}:=a^{15/8}\sum_{x\in\mathcal{C}}\delta_x$ be the normalized counting measure of $\mathcal{C}$. By Theorem 8.2 of \cite{CCK17},
\[\{\mu^a_{\mathcal{C}}:\mathcal{C}\in\mathscr{C}(D^a,f,\cdot)\}\Longrightarrow \{\mu^0_{\mathcal{C}}:\mathcal{C}\in\mathscr{C}(D,f,\cdot)\},\]
where $\Longrightarrow$ denotes convergence in distribution and the right-hand side is a collection of measures obtained from the scaling limit; here the topology of convergence is defined by a metric $dist$; for two collections, $\mathscr{S}$ and $\mathscr{S}^{\prime}$, of measures on $D$,
\begin{equation}\label{eq:metric}
dist(\mathscr{S}, \mathscr{S}^{\prime}):=\inf\{\epsilon>0: \forall \mu \in \mathscr{S}~\exists \nu \in \mathscr{S}^{\prime} \text{ s.t. }d_P(\mu, \nu)\leq\epsilon \text{ and vice versa}\},
\end{equation}
where $d_P$ is the Prokhorov distance. For $h\geq 0$, let $E^0_{D,f,h}$ be the expectation with respect to the continuum random field $\Phi^{h}_D$.

Before stating the next theorem, we first extend the family of random variables $\{\mu^0_{\mathcal{C}}\}$ from magnetic field $h=0$ to $h>0$. Letting $(\Omega,\mathscr{F},\mathbb{P}^0_{D,f,0})$ denote the probability space for $\{\mu^0_{\mathcal{C}}\}$, we define a new ``tilted'' probability measure $\mathbb{P}^0_{D,f,h}$ by
\begin{equation}\label{eq:RNC}
\frac{d\mathbb{P}^0_{D,f,h}}{d\mathbb{P}^0_{D,f,0}}=\frac{\prod_{\mathcal{C}\in\mathscr{C}(D,f,\cdot)}\cosh\left(h\mu^0_{\mathcal{C}}(D)\right)}
{\mathbb{E}^0_{D,f,0}\left(\prod_{\mathcal{C}\in\mathscr{C}(D,f,\cdot)}\cosh\left(h\mu^0_{\mathcal{C}}(D)\right)\right)},
\end{equation}
where $\mu^{0}_{\mathcal{C}}(D) = \int_{D} d\mu^{0}_{\mathcal{C}}$.
The finiteness of the expectation in \eqref{eq:RNC} is proved in Proposition~\ref{prop:mgfcon} below. Then for $h\geq 0$, we define $\mu^0_{\mathcal{C},h}=\mu^0_{\mathcal{C}}$ as a function of $\omega$, but on the tilted space $(\Omega,\mathscr{F},\mathbb{P}^0_{D,f,h})$. We now associate with the clusters $\mathcal{C}\in\mathscr{C}(D,f,\cdot)$, independent uniform $(0,1)$ random variables $U_{\mathcal{C}}$ and define $(\pm1)$-valued variables $S_{\mathcal{C},h}$ by
\begin{equation}
S_{\mathcal{C},h}=
\begin{cases}
      +1, & \text{if } U_{\mathcal{C}}\leq \left(1+\tanh(h\mu^0_{\mathcal{C}}(D))\right)/2,\\
      -1, & \text{otherwise}.
\end{cases}
\end{equation}
We remark that the uniform random variables $U_{\mathcal{C}}$ enable us to couple the models for different values of $h$.
Let $\mathbb{E}^0_{D,f,0}$ be the expectation with respect to $\mathbb{P}^0_{D,f,0}$.
With these definitions, we have the following representation for the near-critical magnetization field $\Phi^h_D$.

\begin{theorem}\label{thm2}
Suppose $D$ is a simply-connected bounded domain in $\mathbb{R}^2$ with piecewise smooth boundary; then
\begin{equation}\label{eq:thm21}
\Phi^h_D\overset{d}=\sum_{\mathcal{C}\in\mathscr{C}(D,f,\cdot)} S_{\mathcal{C},h}\; d\mu^0_{\mathcal{C},h},
\end{equation}
where $\overset{d}=$ means equal in distribution as generalized random fields with test function space $C^{\infty}(\bar{D})$. Indeed, for $\tilde{f}\in C^{\infty}(\bar{D})$,
\begin{equation}\label{eq:thm22}
E^0_{D,f,h}e^{\Phi^h_D(\tilde{f})}=\frac{\mathbb{E}^0_{D,f,0}\left\{\prod_{\mathcal{C}\in\mathscr{C}(D,f,\cdot)}\left[\cosh\left(h\mu^0_{\mathcal{C}}(D)\right)E_{U_{\mathcal{C}}}(e^{S_{\mathcal{C},h}\mu^0_{\mathcal{C}}(\tilde{f})})\right]\right\}}
{\mathbb{E}^0_{D,f,0}\left\{\prod_{\mathcal{C}\in\mathscr{C}(D,f,\cdot)}\cosh\left(h\mu^0_{\mathcal{C}}(D)\right)\right\}},
\end{equation}
where
\[E_{U_{\mathcal{C}}}(e^{S_{\mathcal{C},h}\mu^0_{\mathcal{C}}(\tilde{f})})=\frac{1+\tanh(h\mu^0_{\mathcal{C}}(D))}{2}e^{\mu^0_{\mathcal{C}}(\tilde{f})}+\frac{1-\tanh(h\mu^0_{\mathcal{C}}(D))}{2}e^{-\mu^0_{\mathcal{C}}(\tilde{f})}.\]
\end{theorem}
\begin{remark}\label{rem:RN}
The Radon-Nikodym derivative \eqref{eq:RNC} can be shown to be the limit in $L_1$ of the corresponding lattice expressions and thus the continuum FK measure $\mathbb{P}^0_{D,f,h}$ is the weak limit of the lattice FK measures $\mathbb{P}^a_{D,f,h}$ as $a\downarrow 0$.
\end{remark}
\begin{remark}
Theorem \ref{thm2} and Remark \ref{rem:RN} can be extended to different boundary conditions on $D$ besides free and to the full plane field $\Phi^h_{\mathbb{R}^2}$. In the full plane case, one can replace a constant magnetic field by one which is zero outside $[-L,L]^2$, using the nonconstant field representation of the Appendix. The measure will only converge to a full plane measure weakly as $L\rightarrow\infty$. In the full plane, also the lattice FK measure will only converge weakly as $a\downarrow 0$.
\end{remark}

The next theorem is about the moment generating function for the largest renormalized cluster area.
\begin{theorem}\label{thm3}
Suppose $D$ is a simply-connected and bounded domain in $\mathbb{R}^2$ with piecewise smooth boundary. Then for any $h\geq0$ and $t\geq 0$
\begin{equation}
\mathbb{E}^0_{D,f,h}\left(\exp{\left(t\max_{\mathcal{C}\in\mathscr{C}(D,f,\cdot)}\mu^0_{\mathcal{C}}(D)\right)}\right)\leq \tilde{C}_2e^{\tilde{C}_3(t+h)^{16/15}},
\end{equation}
where $\tilde{C}_2, \tilde{C}_3\in(0,\infty)$ only depend on $D$.
\end{theorem}
\begin{remark}
This bound on the moment generating function shows (by an exponential Chebyshev inequality) that the maximum cluster area has a tail decaying like $e^{-\tilde{C}x^{16}}$ for some constant $\tilde{C}\in(0,\infty)$.
\end{remark}
\begin{remark}
Propositions \ref{prop:mgfbd} and \ref{prop:mgfbdh} (see also Proposition \ref{prop:mgfcbd}) in Section \ref{sec:coupling} provide a lattice analogue to Theorem \ref{thm3}, but where $t^{16/15}$ is replaced by $t^2$; the upper bound there is uniform as the lattice spacing $a\downarrow0$. See \cite{Kis14} for related results about critical Bernoulli percolation.
\end{remark}

We conclude this section with a theorem that improves the result of \cite{CGN14} that on $\mathbb{Z}^2$, for small $H$,
\begin{equation}
B_1H^{1/15}\leq \langle \sigma_0\rangle_{1,H}\leq B_2 H^{1/15}
\end{equation}
for some $B_1,B_2\in (0,\infty)$. The proof is so short that we include it here in this section.
\begin{theorem}\label{thm4}
There exists $B\in (0,\infty)$ such that
\begin{equation}
\lim_{H\downarrow 0}\frac{\langle \sigma_0\rangle_{1,H}}{H^{1/15}}=B.
\end{equation}
\end{theorem}
\begin{proof}
Letting $a=H^{8/15}$, $h=1$, using translation invariance, writing $\Phi^{a,h}$ for $\Phi^{a,h}_{\mathbb{R}^2}$, $1_Q$ for the indicator of $Q=[-1/2,1/2]^2$ and $\mathcal{N}(a)$ for the cardinality of $a\mathbb{Z}^2\cap Q$, one has
\begin{equation}
\langle \Phi^{a,1}(1_Q)\rangle_{a,1}=a^{15/8}\mathcal{N}(a)\langle\sigma_0\rangle_{1,H=a^{15/8}}=H^{-1/15}a^2\mathcal{N}(a)\langle\sigma_0\rangle_{1,H}.
\end{equation}
Since $\mathcal{N}(a)/(1/a)^2\rightarrow 1$ as $a\downarrow 0$, it follows that
\begin{equation}\label{eq:mag}
\lim_{H\downarrow 0}\frac{\langle\sigma_0\rangle_{1,H}}{H^{1/15}}=\lim_{a\downarrow 0}\langle \Phi^{a,1}(1_Q)\rangle_{a,1}=E^0_{h=1}\left(\Phi^{h=1}(1_Q)\right),
\end{equation}
where we have dropped the subscript $D$ in $\Phi^h_{D}$ when $D=\mathbb{R}^2$. The existence of the second limit of \eqref{eq:mag} follows from convergence in distribution of $\Phi^{a,1}$ (see Theorem 1.4 of \cite{CGN16}) and moment generating function bounds (see Proposition 3.5 of \cite{CGN15}).
\end{proof}

\section{Preliminary definitions and results}\label{secpre}
In this section, we give the basic definitions and properties of the Ising model and its coupling to the FK random cluster model. This basically follows the presentation in \cite{CJN17} which we repeat here to make this paper self-contained.
\subsection{Ising model and FK percolation}
In this subsection, our definitions and terminology (especially after the ghost vertex is introduced below) follow those of \cite{Ale98}. With vertex set $a\mathbb{Z}^2$, we write $a\mathbb{E}^2$ for the set of nearest neighbour edges of $a\mathbb{Z}^2$. For any finite $D\subseteq \mathbb{R}^2$, let $D^a:=a\mathbb{Z}^2\cap D$ be the set of points of $a\mathbb{Z}^2$ in $D$, and call it the \textit{$a$-approximation} of $D$. For $\Lambda\subseteq a\mathbb{Z}^2$, define $\Lambda^C:=a\mathbb{Z}^2\setminus\Lambda$,
\[\partial_{in}\Lambda:=\{z\in a\mathbb{Z}^2: z\in \Lambda, z \text{ has a nearest neighbor in }\Lambda^C \}, \]
\[\partial_{ex}\Lambda:=\{z\in a\mathbb{Z}^2: z\notin \Lambda, z \text{ has a nearest neighbor in }\Lambda \}, \]
\[\overline{\Lambda}:=\Lambda\cup \partial_{ex}\Lambda.\]
Let $\mathscr{B}(\Lambda)$ be the set of all edges $\{z,w\}\in a\mathbb{E}^2$ with $z,w\in \Lambda$, and $\overline{\mathscr{B}}(\Lambda)$ be the set of all edges $\{z,w\}$ with $z$ or $w\in \Lambda$. We will consider the extended graph $G=(V,E)$ where $V=a\mathbb{Z}^2\cup\{g\}$ ($g$ is usually called the \textit{ghost vertex} \cite{Gri67}) and $E$ is the set $a\mathbb{E}^2 \cup \{\{z,g\}:z\in a\mathbb{Z}^2\}$. The edges in $a\mathbb{E}^2$ are called \textit{internal edges} while $\{\{z,g\}:z\in a\mathbb{Z}^2\}$ are called \textit{external edges}. Let $\mathscr{E}(\Lambda)$ be the set of all external edges with an endpoint in $\Lambda$, i.e.,
\[\mathscr{E}(\Lambda):=\left\{\left\{z,g\right\}:z\in \Lambda\right\}.\]

Let $\Lambda_L:=[-L,L]^2$ and $\Lambda^a_{L}$ be its $a$-approximation. The classical Ising model at inverse (critical) temperature $\beta_c$ on $\Lambda^a_{L}$ with boundary condition $\eta\in\{-1,+1\}^{\partial_{ex}\Lambda_L^a}$ and external field $a^{\frac{15}{8}}h\geq 0$ is the probability measure $P_{\Lambda_L,\eta, h}^{a}$ on $\{-1,+1\}^{\Lambda^a_{L}}$  such that for any $\sigma\in \{-1,+1\}^{\Lambda^a_{L}}$,
\begin{equation}\label{eqIsingdef}
P_{\Lambda_L,\eta, h}^{a}(\sigma)=\frac{1}{Z^a_{\Lambda_L,\eta,h}}e^{\beta_c\sum_{\{u,v\}}\sigma_u\sigma_v+\beta_c\sum_{\{u,v\}:u\in\Lambda_L^a, v\in\partial_{ex} \Lambda_L^a}\sigma_u\eta_v+a^{15/8}h\sum_{u\in\Lambda^a_{L}}\sigma_u},
\end{equation}
where the first sum is over all nearest neighbor pairs (i.e., $|u-v|=a$) in $\Lambda^a_{L}$, and $Z^a_{\Lambda_L,\eta,h}$ is the partition function (which is the normalization constant needed to make this a probability measure). $P^a_{\Lambda_L,f, h}$ denotes the probability measure with free boundary conditions --- i.e., where we omit the second sum in \eqref{eqIsingdef}. $P^a_{\Lambda_L,+,h}$ (respectively, $P^a_{\Lambda_L,-,h}$) denotes the probability measure with plus (respectively, minus) boundary condition, i.e., $\eta\equiv +1$ (respectively, $\eta\equiv -1$) in \eqref{eqIsingdef}. Below we will also consider Ising measures $P^a_{D,\rho,h}$ for more general domains $D\subseteq\mathbb{R}^2$, defined in the obvious way.

It is known that $P_{\Lambda_L,\eta, h}^{a}$ has a unique infinite volume limit as $L\rightarrow\infty$, which we denote by $P_h^a$. Note that this limiting measure does not depend on the choice of boundary conditions (see, e.g., Theorem 1 of \cite{Leb72} or the theorem in the appendix of \cite{Rue72}).

The FK (Fortuin and Kasteleyn) percolation model at $\beta_c$ on $\Lambda^a_{L}$ with boundary condition $\rho\in\{0,1\}^{\overline{\mathscr{B}}\left((\Lambda^a_L)^C\right)\cup\mathscr{E}\left((\Lambda^a_L)^C\right)}$ and with external field $a^{\frac{15}{8}}h\geq0$ is the probability measure $\mathbb{P}_{\Lambda_L,\rho,h}^a$ on $\{0,1\}^{\mathscr{B}(\Lambda^a_L)\cup\mathscr{E}(\Lambda^a_L)}$ such that for any $\omega\in\{0,1\}^{\mathscr{B}(\Lambda^a_L)\cup\mathscr{E}(\Lambda^a_L)}$,
\begin{align}\label{eqFKdef}
\mathbb{P}_{\Lambda_L,\rho,h}^a(\omega)&=\frac{2^{\mathcal{K}\left(\Lambda_L^a, (\omega\rho)_{\Lambda^a_L}\right)}}{\hat{Z}^a_{\Lambda_L,\rho,h}}\prod_{e\in\mathscr{B}(\Lambda^a_L)}(1-e^{-2\beta_c})^{\omega(e)}(e^{-2\beta_c})^{1-\omega(e)}\nonumber\\
&\quad\times\prod_{e\in\mathscr{E}(\Lambda^a_L)}(1-e^{-2a^{15/8}h})^{\omega(e)}(e^{-2a^{15/8}h})^{1-\omega(e)},
\end{align}
where $(\omega\rho)_{\Lambda^a_L}$ denotes the configuration which coincides with $\omega$ on $\mathscr{B}(\Lambda^a_L)\cup\mathscr{E}(\Lambda^a_L)$ and with $\rho$ on $\overline{\mathscr{B}}\left((\Lambda^a_L)^C\right)\cup\mathscr{E}\left((\Lambda^a_L)^C\right)$, $\mathcal{K}\left(\Lambda_L^a, (\omega\rho)_{\Lambda^a_L}\right)$ denotes the number of clusters in $(\omega\rho)_{\Lambda^a_L}$ which intersect $\Lambda_L^a$ and do not contain $g$, and $\hat{Z}^a_{\Lambda_L,\rho,h}$ is the partition function. An edge $e$ is said to be \textit{open} if $\omega(e)=1$, otherwise it is said to be \textit{closed}. $\mathbb{P}_{\Lambda_L,\rho,h}^a$ is also called the \textit{random-cluster measure} (with cluster weight $q=2$) at $\beta_c$ on $\Lambda^a_{L}$ with boundary condition $\rho$ with external field $a^{\frac{15}{8}}h\geq 0$. $\mathbb{P}^a_{\Lambda_L,f, h}$ (respectively, $\mathbb{P}^a_{\Lambda_L,\bar{w},h}$) denotes the probability measure with free (respectively, wired) boundary conditions, i.e., $\rho\equiv 0$ (respectively, $\rho\equiv 1$) in \eqref{eqFKdef}. In this paper, we use the notation $\mathbb{P}^a_{\Lambda_L,w,h}$ for the boundary condition $\rho$ with $\rho|_{\overline{\mathscr{B}}\left((\Lambda^a_L)^C\right)}\equiv 1$ and $\rho|_{\mathscr{E}\left((\Lambda^a_L)^C\right)}\equiv0$ where $\rho|_{\overline{\mathscr{B}}\left((\Lambda^a_L)^C\right)}$ (respectively, $\rho|_{\mathscr{E}\left((\Lambda^a_L)^C\right)}$) is the restriction of $\rho$ to $\overline{\mathscr{B}}\left((\Lambda^a_L)^C\right)$ (respectively, $\mathscr{E}\left((\Lambda^a_L)^C\right)$). In \eqref{eqFKdef}, we always assume $\rho|_{\mathscr{E}\left((\Lambda^a_L)^C\right)}\equiv0$ when $h=0$. Below we will also consider FK measures $\mathbb{P}^a_{D,\rho,h}$ for more general domains $D\subseteq\mathbb{R}^2$, defined in the obvious way.

It is also known that $\mathbb{P}_{\Lambda_L,\rho,h}^a$ has a unique infinite volume limit as $L\rightarrow\infty$, which we denote by $\mathbb{P}_h^a$. Again this limiting measure does not depend on the choice of boundary conditions. The reader may refer to \cite{Gri06} for more details in the case $h=0$; the proof for $h>0$ is similar.

Suppose $D\subseteq \mathbb{R}^2$ is bounded. For any $\omega\in \{0,1\}^{\mathscr{B}(D^a)}$, let $\mathscr{C}(D^a,\rho,\omega)$ denote the set of clusters of $(\omega\rho)_{D^a}$ which intersect $D^a$. For $\mathcal{C}\in \mathscr{C}(D^a,\rho,\omega)$, let $|\mathcal{C}|$ denote the number of vertices in $\mathcal{C}\cap D^a$. Then the marginal of $\mathbb{P}^a_{D,\rho,h}$ on $\mathscr{B}(D^a)$ (see, e.g., pp. 447-448 of \cite{Ale98}) is
\begin{equation}\label{eq:FKm}
\tilde{\mathbb{P}}^a_{D,\rho,h}(\omega)=\frac{(1-e^{-2\beta_c})^{o(\omega)}(e^{-2\beta_c})^{c(\omega)}\prod_{\mathcal{C}\in \mathscr{C}(D^a,\rho,\omega), g\notin \mathcal{C}}\left(1+e^{-2ha^{15/8}|\mathcal{C}|}\right)}{\tilde{Z}^a_{D^a,\rho,h}},
\end{equation}
where $\tilde{Z}^a_{D^a,\rho,h}$ is the partition function, and $o(\omega)$ and $c(\omega)$ denote the number of open and closed edges of $\omega$ respectively. In this paper, we only consider $\tilde{\mathbb{P}}^a_{D,\rho,h}$ with $\rho$ satisfying $\rho|_{\mathscr{E}\left((D^a)^C\right)}\equiv0$. In particular, in such a case one can drop the condition $g\notin \mathcal{C}$ in \eqref{eq:FKm}.

\subsection{Basic properties}
The Edwards-Sokal coupling \cite{ES88} couples the Ising model and FK percolation. Let $\hat{\mathbb{P}}_h^a$ be that coupling measure of $P_h^a$ and $\mathbb{P}_h^a$ defined on $\{-1,+1\}^V\times\{0,1\}^E$. The marginal of $\hat{\mathbb{P}}_h^a$ on $\{-1,+1\}^V$ is $P_h^a$, and the marginal of $\hat{\mathbb{P}}_h^a$ on $\{0,1\}^E$ is $\mathbb{P}_h^a$. The conditional distribution of the Ising spin variables given a realization of the FK bond variables can be realized by tossing independent fair coins --- one for each FK-open cluster not containing $g$ --- and then setting $\sigma_x$ for all vertices $x$ in the cluster to $+1$ for heads and $-1$ for tails. For $x$ in the ghost cluster, $\sigma_x= +1$ (for $h>0$). The coupling we just described is for the infinite graph $G=(V,E)$, but it is also valid for more general graphs (infinite and finite). A different coupling for $h\neq 0$ between \emph{internal} FK edges and spin variables is given in Propositions~\ref{propFKcon} and \ref{prop:tanh} below as well as in the Appendix when the magnetic field $h_x$ at vertex $x$ may vary with $x$ and not be a constant $h$. As mentioned above, this coupling already appears implicitly in \cite{CV16}.

For any $u,v\in V$, we write $u\longleftrightarrow v$ for the event that there is a path of FK-open edges that connects $u$ and $v$, i.e., a path $u=z_0, z_1, \ldots, z_n=v$ with $e_i=\{z_i, z_{i+1}\}\in E$ and $\omega(e_i)=1$ for each $0\leq i<n$. For any $u,v\in a\mathbb{Z}^2$, we write $u\overset{a\mathbb{Z}^2}\longleftrightarrow v$ if $u\longleftrightarrow v$ and each vertex on this path is in $a\mathbb{Z}^2$. For any $A, B\subseteq V$, we write $A\longleftrightarrow B$ if there is some $u\in A$ and $v\in B$ such that $u\longleftrightarrow v$. $A\centernot\longleftrightarrow B$ denotes the complement of $A\longleftrightarrow B$. The following identity, immediate from the Edwards-Sokal coupling, is essential.

\begin{lemma}\label{lemES}
\begin{equation}\label{eqES}
\langle\sigma_x;\sigma_y\rangle_{a,h}=\mathbb{P}_h^a(x\longleftrightarrow y)-\mathbb{P}_h^a(x\longleftrightarrow g)\mathbb{P}_h^a(y\longleftrightarrow g).
\end{equation}
\end{lemma}


Let $\mathbb{P}^a:=\mathbb{P}^a_{h=0}$. By standard comparison inequalities for FK percolation (Proposition 4.28 in \cite{Gri06}), one has
\begin{lemma}\label{lem3}
For any $h\geq 0$, $\mathbb{P}_h^a$ stochastically dominates $\mathbb{P}^a$.
\end{lemma}

The following lemma is about the one-arm exponent for FK percolation with $h=0$. It is an immediate consequence of Lemma 5.4 of \cite{DCHN11}.
\begin{lemma}[\cite{DCHN11}]\label{lem1arm}
There exist constants $\tilde{C}_1,C_1$, independent of $a$, such that for each $a\in(0,1]$ and any boundary condition $\rho\in\{0,1\}^{\overline{\mathscr{B}}\left((\Lambda^a_1)^C\right)\cup\mathscr{E}\left((\Lambda^a_1)^C\right)}$,
\[\tilde{C}_1 a^{1/8}\leq \mathbb{P}^a_{\Lambda_1,\rho,h=0}(0\longleftrightarrow\partial_{in}\Lambda_1^a)\leq C_1a^{1/8}.\]
\end{lemma}

Let $Q:=\Lambda_{1/2}$ be the unit square centered at the origin. Let $E^a_{Q,+,0}$ be the expectation with respect to $P^a_{Q,+,0}$. Let $m^a_Q$ be the renormalized magnetization in $Q$ defined by
\[m^a_Q:=a^{15/8}\sum_{x\in Q^a}\sigma_x.\]
Then we have
\begin{lemma}[\cite{CGN15}]\label{lem:mgf}
There exists $C_5\in (0,\infty)$ such that for any $a>0$ and $t>0$
\[E^a_{Q,+,0}(e^{tm^a_Q})\leq e^{C_5(t+t^2)}.\]
\end{lemma}
\begin{proof}
This follows from the proof of Proposition 3.5 in \cite{CGN15}.
\end{proof}

\section{Couplings of FK and Ising variables}\label{sec:coupling}
\subsection{A coupling for $h>0$}
In the Appendix, we discuss the coupling of FK and Ising variables for general finite graphs with general non-negative magnetic field profiles---see also \cite{CV16}. In this subsection, we focus on the critical FK measure in a finite domain with constant magnetic field, but general boundary conditions.
The following two propositions are generalizations of Lemma 4 and Proposition 1 in \cite{CJN17}.
\begin{proposition}\label{propFKcon}
Suppose $\rho\in\{0,1\}^{\overline{\mathscr{B}}\left(\left(D^a\right)^C\right)\cup\mathscr{E}\left(\left(D^a\right)^C\right)}$ with $\rho|_{\mathscr{E}\left((D^a)^C\right)}\equiv0$. Then the Radon-Nikodym derivative of $\tilde{\mathbb{P}}_{D,\rho,h}^a$ with respect to $\mathbb{P}_{D,\rho,0}^a$ is
\begin{equation}\label{eqRN}
\frac{d\tilde{\mathbb{P}}_{D,\rho,h}^a}{d\mathbb{P}_{D,\rho,0}^a}(\omega)=\frac{\prod_{\mathcal{C}\in \mathscr{C}(D^a,\rho,\omega)}\cosh(ha^{15/8}|\mathcal{C}|)}{\mathbb{E}_{D,\rho,0}^a\left[\prod_{\mathcal{C}\in \mathscr{C}(D^a,\rho,\cdot)}\cosh(ha^{15/8}|\mathcal{C}|)\right]},\, \omega\in\{0,1\}^{\mathscr{B}(D^a)}
\end{equation}
where $\mathbb{E}_{D,\rho,0}^a$ is the expectation with respect to $\mathbb{P}_{D,\rho,0}^a$.
\end{proposition}

\begin{remark}
To be more precise, $\mathbb{P}_{D,\rho,0}^a$ and $\mathbb{E}_{D,\rho,0}^a$ in \eqref{eqRN}  should be $\tilde{\mathbb{P}}_{D,\rho,0}^a$ and $\tilde{\mathbb{E}}_{D,\rho,0}^a$ respectively. But for ease of notation, we drop the $\sim$ notation when $h=0$. This should cause no confusion.
\end{remark}
\begin{proof}
By \eqref{eqFKdef} and \eqref{eq:FKm}, we have
\begin{eqnarray*}
\frac{\tilde{\mathbb{P}}^a_{D,\rho,h}(\omega)}{\mathbb{P}^a_{D,\rho,0}(\omega)}&=&\frac{\hat{Z}^a_{D,\rho,0}\prod_{\mathcal{C}\in\mathscr{C}(D^a,\rho,\omega)}(1+e^{-2ha^{15/8}|\mathcal{C}|})}{\tilde{Z}^a_{D,\rho,h}\prod_{\mathcal{C}\in\mathscr{C}(D^a,\rho,\omega)}2}\\
&=&\frac{\hat{Z}^a_{D,\rho,0}}{\tilde{Z}^a_{D,\rho,h}}\prod_{\mathcal{C}\in\mathscr{C}(D^a,\rho,\omega)}\frac{1+e^{-2ha^{15/8}|\mathcal{C}|}}{2}\\
&=&\frac{e^{-ha^{15/8}|D^a|}\hat{Z}^a_{D,\rho,0}}{\tilde{Z}^a_{D,\rho,h}}\prod_{\mathcal{C}\in\mathscr{C}(D^a,\rho,\omega)}\cosh(ha^{15/8}|\mathcal{C}|),
\end{eqnarray*}
where $|D^a|$ is the total number of vertices in $D^a$. Since $\frac{e^{-ha^{15/8}|D^a|}\hat{Z}^a_{D,\rho,0}}{\tilde{Z}^a_{D,\rho,h}}$ only depends on $a,D,\rho,h$ but not on $\omega$, the proposition follows.
\end{proof}

Let $\hat{\mathbb{P}}^a_{D,\rho,h}$ be the Edwards-Sokal coupling of $\mathbb{P}^a_{D,\rho,h}$ and its corresponding Ising measure. For any $\mathcal{C}\in \mathscr{C}(D^a,\rho,\omega)$, let $\sigma(\mathcal{C})$ be the spin value of the cluster assigned by the coupling. Then we have
\begin{proposition}\label{prop:tanh}
Let $\rho\in\{0,1\}^{\overline{\mathscr{B}}\left(\left(D^a\right)^C\right)\cup\mathscr{E}\left(\left(D^a\right)^C\right)}$. For  $\omega\in\{0,1\}^{\mathscr{B}(D^a)}$, suppose $\mathscr{C}(D^a,\rho,\omega)=\{\mathcal{C}_1,\mathcal{C}_2,\ldots\}$ where the $\mathcal{C}_i$'s are distinct. Then for any $\mathcal{C}_i\in \mathscr{C}(D^a,\rho,\omega)$ with $g\notin \mathcal{C}_i$
\begin{align}
&\mathbb{P}_{D,\rho,h}^a(\mathcal{C}_i\longleftrightarrow g|\omega)=\tanh(ha^{15/8}|\mathcal{C}_i|),\label{eqtanh}\\
&\hat{\mathbb{P}}_{D,\rho,h}^a(\sigma(\mathcal{C}_i)=+1|\omega)=\tanh(ha^{15/8}|\mathcal{C}_i|)+\frac{1}{2}\left(1-\tanh(ha^{15/8}|\mathcal{C}_i|)\right),\\
&\hat{\mathbb{P}}_{D,\rho,h}^a(\sigma(\mathcal{C}_i)=-1|\omega)=\frac{1}{2}\left(1-\tanh(ha^{15/8}|\mathcal{C}_i|)\right).
\end{align}
Moreover, conditioned on $\omega$, the events $\{\mathcal{C}_i\longleftrightarrow g\}$ are mutually independent and the events $\{\sigma(\mathcal{C}_i)=+1\}$ are mutually independent.
\end{proposition}
\begin{proof}
For each $\omega\in\{0,1\}^{\mathscr{B}(D^a)}$, one sees from \eqref{eq:FKm} (note that $\mathbb{P}_{D,f,h}^a(\omega)$ and $\tilde{\mathbb{P}}_{D,f,h}^a(\omega)$ are equal) that
\begin{align}\label{eqFKweight}
\mathbb{P}_{D,f,h}^a(\omega)&\propto\left(1-e^{-2\beta_c}\right)^{o(\omega)}\left(e^{-2\beta_c}\right)^{c(\omega)}\nonumber\\&\quad \times \prod_{\mathcal{C}\in \mathscr{C}(D^a,\rho,\omega), g\notin \mathcal{C}}\left((1-e^{-2ha^{15/8}|\mathcal{C}|})+2e^{-2ha^{15/8}|\mathcal{C}|}\right).
\end{align}
So for any $\mathcal{C}_i, \mathcal{C}_j\in \mathscr{C}(D^a,\rho,\omega)$ with $g\notin \mathcal{C}_i$, $g\notin \mathcal{C}_j$ and $i\neq j$,
\[\mathbb{P}_{D,f,h}^a(\mathcal{C}_i\longleftrightarrow g|\omega)=\frac{1-e^{-2ha^{15/8}|\mathcal{C}_i|}}{(1-e^{-2ha^{15/8}|\mathcal{C}_i|})+2e^{-2ha^{15/8}|\mathcal{C}_i|}}=\tanh(ha^{15/8}|\mathcal{C}_i|),\]
\[\mathbb{P}_{D,f,h}^a(\mathcal{C}_i\longleftrightarrow g, \mathcal{C}_j\longleftrightarrow g|\omega)=\tanh(ha^{15/8}|\mathcal{C}_i|)\tanh(ha^{15/8}|\mathcal{C}_j|),\]
with a similar product expression for the intersection of three or more of the events $\{\mathcal{C}_i\longleftrightarrow g\}$. So conditioned on $\omega$, these events are mutually independent. The rest of the proposition follows directly from the Edwards-Sokal coupling.
\end{proof}

\subsection{FK measure without external field}
Let $m^a_D$ be the renormalized magnetization in $D$, i.e.,
\[m^a_D:=a^{15/8}\sum_{x\in D^a}\sigma_x.\]
By the Edwards-Sokal coupling, for each FK measure $\mathbb{P}^a_{D,\rho,0}$, there is a corresponding Ising measure which is denoted by $P^a_{D,\rho,0}$. Let $\mathbb{E}^a_{D,\rho,0}$ (respectively, $E^a_{D,\rho,0}$) be the expectation with respect to $\mathbb{P}^a_{D,\rho,0}$ (respectively, $P^a_{D,\rho,0}$). Recall that when $h=0$ we always assume $\rho\in\{0,1\}^{\overline{\mathscr{B}}\left(\left(D^a\right)^C\right)\cup\mathscr{E}\left(\left(D^a\right)^C\right)}$ with $\rho|_{\mathscr{E}\left((D^a)^C\right)}\equiv0$. Then we have
\begin{proposition}\label{prop:mgf}
For any $a>0$ and $h>0$, we have
\begin{equation}\label{eq:mgf}
1\leq \mathbb{E}_{D,\rho,0}^a\left[\prod_{\mathcal{C}\in \mathscr{C}(D^a,\rho,\cdot)}\cosh(ha^{15/8}|\mathcal{C}|)\right]=E^a_{D,\rho,0}\left(e^{hm^a_D}\right)\leq E^a_{D,+,0}\left(e^{hm^a_D}\right).
\end{equation}
\end{proposition}
\begin{proof}
The leftmost inequality in \eqref{eq:mgf} is trivial since $\cosh(r)\geq 1$ for any $r\in\mathbb{R}$. By the Edwards-Sokal coupling (see, e.g., (3.2) in \cite{CGN16})
\begin{align*}
E^a_{D,\rho,0}\left(e^{hm^a_D}\right)&=\mathbb{E}^a_{D,\rho,0}\left[\prod_{\mathcal{C}\in \mathscr{C}(D^a,\rho,\cdot)}\left(\frac{1}{2}e^{ha^{15/8}|\mathcal{C}|}+\frac{1}{2}e^{-ha^{15/8}|\mathcal{C}|}\right)\right]\\
&=\mathbb{E}_{D,\rho,0}^a\left[\prod_{\mathcal{C}\in \mathscr{C}(D^a,\rho,\cdot)}\cosh(ha^{15/8}|\mathcal{C}|).\right]
\end{align*}
The last inequality in \eqref{eq:mgf} follows from the FKG inequality.
\end{proof}

If $D$ is a simply-connected domain and $\rho$ is either free or wired, then Theorem 2.6 of \cite{CGN15} says $m^a_D$ converges weakly to a continuum magnetization variable $m_D$ (Theorem 2.6 is for a dyadic square but the same proof applies to a general simply-connected domain). Then by Corollary 3.8 of \cite{CGN15}, we have
\begin{equation}\label{eq:mgflimit}
\lim_{a\downarrow 0}E^a_{D,\rho,0}\left(e^{hm^a_D}\right)=E_{D,\rho,0}\left(e^{hm_D}\right),
\end{equation}
which yields the following proposition.
\begin{proposition}
If $D$ is a simply-connected domain and $\rho$ is either free or wired, then
\[\lim_{a\downarrow 0}\mathbb{E}_{D,\rho,0}^a\left[\prod_{\mathcal{C}\in \mathscr{C}(D^a,\rho,\cdot)}\cosh(ha^{15/8}|\mathcal{C}|)\right]=E_{D,\rho,0}\left(e^{hm_D}\right)\leq E_{D,+,0}\left(e^{hm_D}\right).\]
\end{proposition}
\begin{proof}
The equality follows from \eqref{eq:mgf} and \eqref{eq:mgflimit} while the inequality follows from the FKG inequality.
\end{proof}

Recall that $Q$ is the unit square centered at the origin. For a configuration $\omega$ sampled from the measure $\mathbb{P}^a_{Q,w,0}$, let $\mathcal{C}_0(\omega)$ be the boundary cluster (note that there is only one such cluster). For a configuration $\omega$ sampled from $\mathbb{P}^a_{Q,\rho,0}$, let $\mathcal{A}_{max}(\omega)$ denote the maximum number of vertices of any FK-open cluster. Let $A_0(\omega):=a^{15/8}|\mathcal{C}_0(\omega)|$ and $A_{max}(\omega):=a^{15/8}\mathcal{A}_{max}(\omega)$ be the corresponding renormalized ``areas''. Then we have
\begin{proposition}\label{prop:mgfbd}
For any $a>0$ and $t>0$, we have
\[\mathbb{E}^a_{Q,w,0}\left(e^{tA_0}\right)\leq 2 e^{C_5(t+t^2)},~\mathbb{E}^a_{Q,\rho,0}\left(e^{tA_{max}}\right)\leq 2 e^{C_5(t+t^2)} \text{ for any }\rho,\]
where $C_5$ is as in Lemma \ref{lem:mgf}.
\end{proposition}
\begin{proof}
We only prove the second inequality since the first follows from the second. By Proposition \ref{prop:mgf} and Lemma \ref{lem:mgf}, we have
\begin{align}
\mathbb{E}^a_{Q,\rho,0}\left(e^{tA_{max}}\right)&\leq 2\mathbb{E}_{Q,\rho,0}^a\left[\prod_{\mathcal{C}\in \mathscr{C}(Q^a,\rho,\cdot)}\cosh(ta^{15/8}|\mathcal{C}|)\right]\nonumber\\
&\leq 2E^a_{Q,+,0}\left(e^{tm^a_D}\right)\leq 2 e^{C_5(t+t^2)},
\end{align}
where the first inequality follows from $e^r\leq 2\cosh(r)$ and $\cosh(r)\geq 1$ for any $r\in\mathbb{R}$.
\end{proof}

\subsection{FK measure with external field}
In this subsection we present three propositions concerning the moment generating function of cluster size and one-arm events. They will be used in Section \ref{sec:correlation} below.

For a configuration $\omega$ from the measure $\tilde{\mathbb{P}}^a_{Q,w,h}$, we again let $\mathcal{C}_0(\omega)$ be the boundary cluster and $A_0(\omega):=a^{15/8}|\mathcal{C}_0(\omega)|$ be the corresponding renormalized area. For a configuration $\omega$ from the measure $\mathbb{P}^a_{Q,w,0}$, let $\mathscr{C}(D^a,w,\omega)=\{\mathcal{C}_0,\mathcal{C}_1,\mathcal{C}_2,\ldots\}$ where $\mathcal{C}_0$ is the boundary cluster. Define $A_i(\omega):=a^{15/8}|\mathcal{C}_i|$ for each $i\geq 0$. Let $\tilde{\mathbb{E}}^a_{Q,w,h}$ be expectation with respect to $\tilde{\mathbb{P}}^a_{Q,w,h}$.
\begin{proposition}\label{prop:mgfbdh}
For any $a>0$, $h\geq 0$ and $t>0$, we have
\[\tilde{\mathbb{E}}^a_{Q,w,h}\left(e^{tA_0}\right)\leq  2 e^{C_5\left((t+h)^2+(t+h)\right)}\]
where $C_5$ is as in Lemma \ref{lem:mgf}.
\end{proposition}
\begin{proof}
Proposition \ref{propFKcon} implies
\begin{align*}
\tilde{\mathbb{E}}^a_{Q,w,h}\left(e^{tA_0}\right)&=\frac{\mathbb{E}^a_{Q,w,0}\left(e^{tA_0}\prod_{\mathcal{C}\in \mathscr{C}(Q^a,w,\cdot)}\cosh(ha^{15/8}|\mathcal{C}|)\right)}{\mathbb{E}_{Q,w,0}^a\left[\prod_{\mathcal{C}\in \mathscr{C}(Q^a,w,\cdot)}\cosh(ha^{15/8}|\mathcal{C}|)\right]}\\
&\leq \mathbb{E}^a_{Q,w,0}\left(e^{tA_0}\prod_{i\geq 0} \cosh(hA_i)\right)\\
&\leq 2\mathbb{E}^a_{Q,w,0}\left(\prod_{i\geq 0} \cosh\left((t+h)A_i\right)\right),
\end{align*}
where the last inequality follows from $e^{tr}\cosh(hr)\leq 2\cosh((t+h)r)$ and $\cosh(hs)\leq \cosh((t+h)s)$, valid for any $r,s\geq 0$. The proof is completed by using Proposition \ref{prop:mgf} and Lemma \ref{lem:mgf}.
\end{proof}
The following proposition is about the one-arm event for $\tilde{\mathbb{P}}^a_{Q,w,h}$.
\begin{proposition}\label{prop:onearm}
For any $a>0$ and $h\geq0$, we have
\[\tilde{\mathbb{P}}^a_{Q,w,h}(0\longleftrightarrow\partial_{in} Q^a)\leq C_7(h)a^{1/8},\]
where $C_7(h)\in(0,\infty)$ only depends on $h$.
\end{proposition}
\begin{proof}
The $h=0$ case follows from Lemma \ref{lem1arm}, so we assume $h>0$ in the rest of the proof. Let $E^a_{Q,+,h}$ be the expectation with respect to $P^a_{Q,+,h}$. Then, by the Edwards-Sokal coupling and the FKG inequality (recall that the subscripts $\bar{w}$ and $w$ refer to wired boundary conditions, see the discussion after \eqref{eqFKdef}),
\begin{align}\label{eq:esm}
E^a_{Q,+,h}(\sigma_0)&=\mathbb{P}^a_{Q,\bar{w},h}(0\longleftrightarrow g)\geq\mathbb{P}^a_{Q,\bar{w},h}(0\overset{a\mathbb{Z}^2}\longleftrightarrow \partial_{in} Q^a)\nonumber\\
&\geq \mathbb{P}^a_{Q,w,h}(0\overset{a\mathbb{Z}^2}\longleftrightarrow \partial_{in} Q^a)=\tilde{\mathbb{P}}^a_{Q,w,h}(0\longleftrightarrow\partial_{in} Q^a).
\end{align}
Let $Q_{1/2}:=[-1/4,1/4]^2$ and $Q^a_{1/2}$ be its a-approximation. Then by the domain Markov property and FKG inequality,
\[E^a_{Q,+,h}(\sigma_0)= E^a_{z+Q,+,h}(\sigma_z)\leq E^a_{Q_{1/2},+,h}(\sigma_z) \text{ for any }z\in Q^a_{1/2}\]
since $Q^a_{1/2}\subseteq z+Q^a$ for each such $z$. Therefore
\begin{equation}\label{eq:spinm}
E^a_{Q,+,h}(\sigma_0)\leq \frac{1}{|Q^a_{1/2}|}\sum_{z\in Q^a_{1/2}}E^a_{Q_{1/2},+,h}(\sigma_z),
\end{equation}
where $|Q^a_{1/2}|$ is the number of vertices in $Q^a_{1/2}$.

Let $m^a_{Q_{1/2},h}:=a^{15/8}\sum_{z\in Q^a_{1/2}}\sigma_z$. Using the Radon-Nikodym derivative of $P^a_{Q_{1/2},+,h}$ with respect to $P^a_{Q_{1/2},+,0}$ (see the proof of Proposition 1.5 in \cite{CGN16}),
\begin{equation}\label{eq:mrn}
E^a_{Q_{1/2},+,h}\left(m^a_{Q_{1/2},h}\right)=\frac{E^a_{Q_{1/2},+,0}\left(m^a_{Q_{1/2}}e^{hm^a_{Q_{1/2}}}\right)}{E^a_{Q_{1/2},+,0}\left(e^{hm^a_{Q_{1/2}}}\right)}.
\end{equation}
Note that
\begin{align}\label{eq:mup}
E^a_{Q_{1/2},+,0}\left(m^a_{Q_{1/2}}e^{hm^a_{Q_{1/2}}}\right)&\leq E^a_{Q_{1/2},+,0}\left(e^{m^a_{Q_{1/2}}}e^{hm^a_{Q_{1/2}}}\right)\nonumber\\
&=E^a_{Q_{1/2},+,0}\left(e^{(h+1)m^a_{Q_{1/2}}}\right).
\end{align}
By Jensen's inequality,
\begin{equation}\label{eq:jensen}
E^a_{Q_{1/2},+,0}\left(e^{hm^a_{Q_{1/2}}}\right)\geq e^{h E^a_{Q_{1/2},+,0}\left(m^a_{Q_{1/2}}\right)}\geq 1,
\end{equation}
since $E^a_{Q_{1/2},+,0}\left(m^a_{Q_{1/2}}\right)\geq E^a_{Q_{1/2},f,0}\left(m^a_{Q_{1/2}}\right)=0$ by the FKG inequality.
Combining \eqref{eq:jensen}, \eqref{eq:mup}, \eqref{eq:mrn} and \eqref{eq:spinm}, we get
\[E^a_{Q,+,h}(\sigma_0)\leq \frac{a^{-2}}{|Q^a_{1/2}|}a^{1/8}E^a_{Q_{1/2},+,0}\left(e^{(h+1)m^a_{Q_{1/2}}}\right)\leq C_7(h)a^{1/8},\]
where the last inequality with $C_7(h)<\infty$ follows from  $a^{-2}/|Q^a_{1/2}|\rightarrow 4$ as $a\downarrow 0$ and a similar argument as in the proof of Proposition 3.5 of \cite{CGN15} (see also Lemma \ref{lem:mgf} above).
This and \eqref{eq:esm} complete the proof.
\end{proof}

Next, we will show that the moment generating function of the boundary cluster from $\tilde{\mathbb{P}}^a_{Q,w,h}$ is still finite even after conditioning on the event $\{0\longleftrightarrow \partial_{in} Q^a\}$. Recall the definitions of $A_0, A_1, \dots$ right before Proposition \ref{prop:mgfbdh}.
\begin{proposition}\label{prop:mgfcbd}
For any $a>0$, $h\geq0$ and $t>0$, we have
\[\tilde{\mathbb{E}}^a_{Q,w,h}\left(e^{tA_0}|0\longleftrightarrow \partial_{in} Q^a\right)\leq C_8(t+h)e^{C_5\left(h+h^2+(t+h)+(t+h)^2\right)},\]
where $C_8(t+h)\in(0,\infty)$ only depends on $t+h$ and $C_5$ is the same as in Lemma \ref{lem:mgf}.
\end{proposition}
\begin{proof}
By Proposition \ref{propFKcon},
\begin{align}\label{eq:mgfcbd1}
&\tilde{\mathbb{E}}^a_{Q,w,h}\left(e^{tA_0}1_{\{0\longleftrightarrow \partial_{in} Q^a\}}\right)=\frac{\mathbb{E}^a_{Q,w,0}\left(e^{tA_0}1_{\{0\longleftrightarrow \partial_{in} Q^a\}}\prod_{i\geq 0}\cosh(hA_i)\right)}{\mathbb{E}^a_{Q,w,0}\left(\prod_{i\geq 0}\cosh(hA_i)\right)}\nonumber\\
&\quad\leq 2\mathbb{E}^a_{Q,w,0}\left(\cosh\left((t+h)A_0\right)1_{\{0\longleftrightarrow \partial_{in} Q^a\}}\prod_{i\geq 1}\cosh(hA_i)\right)
\end{align}
since $e^{tA_0}\cosh(hA_0)\leq 2\cosh\left((t+h)A_0\right)$ and the denominator is larger than or equal to~1. Let $\Gamma\subseteq Q^a$ be a possible realization in $Q^a$ (with wired boundary conditions) of the cluster of $0$ (i.e., a lattice animal containing $0$) such that there is a path from $0$ to $\partial_{in}Q^a$ with each vertex on the path in $\Gamma$. Then
\begin{align}\label{eq:mgfcbd2}
&\mathbb{E}^a_{Q,w,0}\left(\cosh\left((t+h)A_0\right)1_{\{0\longleftrightarrow \partial_{in} Q^a\}}\prod_{i\geq 1}\cosh(hA_i)\right)\nonumber\\
&=\sum_{\Gamma} \cosh\left(\left(t+h\right)|\Gamma|\right)\mathbb{P}^a_{Q,w,0}(\mathcal{C}_0=\Gamma)\mathbb{E}^a_{Q,w,0}\left(\prod_{i\geq 1}\cosh(hA_i)|\mathcal{C}_0=\Gamma\right),
\end{align}
where $\mathcal{C}_0$ is the boundary cluster and thus also the cluster of $0$. Define
\[\bar{\Gamma}:=\{\text{edges }e\in Q^a: \text{ at least one endpoint of } e \text{ is in }\Gamma\},\]
so that $\bar{\Gamma}$ includes both the open edges in $\Gamma$ and the closed edges touching $\Gamma$.

Note that $\mathbb{P}^a_{Q,w,0}(\cdot|\mathcal{C}_0=\Gamma)$ is an FK measure on $Q^a\setminus \bar{\Gamma}$ with free boundary conditions. So by Proposition \ref{prop:mgf}, the GKS inequalities \cite{Gri67a,KS68} used three times and Lemma \ref{lem:mgf},
\begin{align}\label{eq:GKS}
&\mathbb{E}^a_{Q,w,0}\left(\prod_{i\geq 1}\cosh(hA_i)|\mathcal{C}_0=\Gamma\right)=\mathbb{E}^a_{Q\setminus\bar{\Gamma},f,0}\left(\prod_{i\geq 1}\cosh(hA_i)\right)\nonumber\\
&= E^a_{Q\setminus\bar{\Gamma},f,0}(e^{hm^a_{Q\setminus\bar{\Gamma}}})\leq E^a_{Q,f,0}(e^{hm^a_{Q\setminus\bar{\Gamma}}})\leq E^a_{Q,f,0}(e^{hm^a_{Q}})\nonumber\\
&\leq E^a_{Q,+,0}(e^{hm^a_{Q}})\leq e^{C_5(h+h^2)},
\end{align}
where the second inequality follows from expanding the exponentials on both sides and noticing that the extra terms from the RHS are non-negative by the GKS inequalities.

Therefore by \eqref{eq:mgfcbd1}, \eqref{eq:mgfcbd2} and \eqref{eq:GKS},
\begin{align}\label{eq:mgfc}
&\tilde{\mathbb{E}}^a_{Q,w,h}\left(e^{tA_0}1_{\{0\longleftrightarrow \partial_{in} Q^a\}}\right)\leq 2 \sum_{\Gamma} \cosh\left(\left(t+h\right)|\Gamma|\right)\mathbb{P}^a_{Q,w,0}(\mathcal{C}_0=\Gamma)e^{C_5(h+h^2)}\nonumber\\
&= 2e^{C_5(h+h^2)} \mathbb{E}^a_{Q,w,0}\left(\cosh\left((t+h)A_0\right)1_{\{0\longleftrightarrow\partial_{in} Q^a\}}\right)\nonumber\\
&\leq 2e^{C_5(h+h^2)}\mathbb{E}^a_{Q,w,0}\left(1_{\{0\longleftrightarrow\partial_{in} Q^a\}}\prod_{i\geq 0}\cosh\left((t+h)A_i\right)\right)\nonumber\\
&= 2e^{C_5(h+h^2)}\mathbb{E}^a_{Q,w,0}\left(\prod_{i\geq 0}\cosh\left((t+h)A_i\right)\right)\tilde{\mathbb{P}}^a_{Q,w,t+h}(0\longleftrightarrow\partial_{in} Q^a),
\end{align}
where the last equality holds because, by Proposition \ref{propFKcon},
\[\tilde{\mathbb{P}}^a_{Q,w,t+h}(0\longleftrightarrow\partial_{in} Q^a)=\frac{\mathbb{E}^a_{Q,w,0}\left(1_{\{0\longleftrightarrow\partial_{in} Q^a\}}\prod_{i\geq 0}\cosh\left((t+h)A_i\right)\right)}{\mathbb{E}^a_{Q,w,0}\left(\prod_{i\geq 0}\cosh\left((t+h)A_i\right)\right)}.\]
Proposition \ref{prop:mgf}, Lemma \ref{lem:mgf} and \eqref{eq:mgfc} imply
\begin{equation}\label{eq:mgfc1}
\tilde{\mathbb{E}}^a_{Q,w,h}\left(e^{tA_0}1_{\{0\longleftrightarrow \partial_{in} Q^a\}}\right)\leq 2\tilde{\mathbb{P}}^a_{Q,w,t+h}(0\longleftrightarrow\partial_{in} Q^a)e^{C_5\left(h+h^2+(t+h)+(t+h)^2\right)}.
\end{equation}
Note that by the FKG inequality and Lemma \ref{lem1arm}
\begin{equation}\label{eq:mgfc2}
\tilde{\mathbb{P}}^a_{Q,w,h}(0\longleftrightarrow\partial_{in} Q^a)\geq \mathbb{P}^a_{Q,w,0}(0\longleftrightarrow\partial_{in} Q^a)\geq \tilde{C}_1 a^{1/8}.
\end{equation}
Hence by \eqref{eq:mgfc1}, \eqref{eq:mgfc2} and Proposition \ref{prop:onearm}
\begin{align*}
\tilde{\mathbb{E}}^a_{Q,w,h}\left(e^{tA_0}|0\longleftrightarrow \partial_{in} Q^a\right)&=\frac{\tilde{\mathbb{E}}^a_{Q,w,h}\left(e^{tA_0}1_{\{0\longleftrightarrow \partial_{in} Q^a\}}\right)}{\tilde{\mathbb{P}}^a_{Q,w,h}(0\longleftrightarrow\partial_{in} Q^a)}\\
&\leq\frac{2\tilde{\mathbb{P}}^a_{Q,w,t+h}(0\longleftrightarrow\partial_{in} Q^a)}{\tilde{\mathbb{P}}^a_{Q,w,h}(0\longleftrightarrow\partial_{in} Q^a)}e^{C_5\left(h+h^2+(t+h)+(t+h)^2\right)}\\
&\leq C_8(t+h)e^{C_5\left(h+h^2+(t+h)+(t+h)^2\right)}
\end{align*}
with $C_8(t+h)=2C_7(t+h)/\tilde{C}_1$.
\end{proof}

\section{A lower bound for the correlation length (upper bound for the mass)}\label{sec:correlation}
In this section, we prove Theorem \ref{thm:upper}. We state and prove several lemmas first. In the first of these, the constant $C_5$ may be taken as in Lemma~\ref{lem:mgf}.

\begin{lemma}\label{lem:pos}
There is some $C_5\in (0,\infty)$ so that for any $a>0$, $h\geq0$, boundary condition $\rho$ on $Q^a$ and event $E\subseteq \{0,1\}^{\mathscr{B}(Q^a)}$,
\[\tilde{\mathbb{P}}^a_{Q,\rho,h}(E)\geq e^{-C_5(h+h^2)}\mathbb{P}^a_{Q,\rho,0}(E).\]
\end{lemma}
\begin{proof}
By Proposition \ref{propFKcon},
\begin{align*}
\tilde{\mathbb{P}}^a_{Q,\rho,h}(E)&=\frac{\sum_{\omega\in E}\mathbb{P}^a_{Q,\rho,0}(\omega)\prod_{i}\cosh(hA_i(\omega))}{\mathbb{E}^a_{Q,\rho,0}\left(\prod_{i}\cosh(A_i)\right)}\geq \frac{\mathbb{P}^a_{Q,\rho,0}(E)}{E^a_{Q,+,0}(e^{hm^a_Q})}\\
&\geq e^{-C_5(h+h^2)}\mathbb{P}^a_{Q,\rho,0}(E),
\end{align*}
where the first inequality follows since $\cosh(r)\geq 1$ for any $r\in\mathbb{R}$ and Proposition \ref{prop:mgf}, and the second inequality follows from Lemma \ref{lem:mgf}.
\end{proof}
\begin{remark}\label{rem:pos}
It is not hard to see that Lemma \ref{lem:pos} holds for more general domains. For example, below we will apply it to the domain $[0,\frac{1}{2}]\times[0,\frac{1}{4}]$.
\end{remark}
\begin{figure}
\begin{center}
\includegraphics[width=0.9\textwidth]{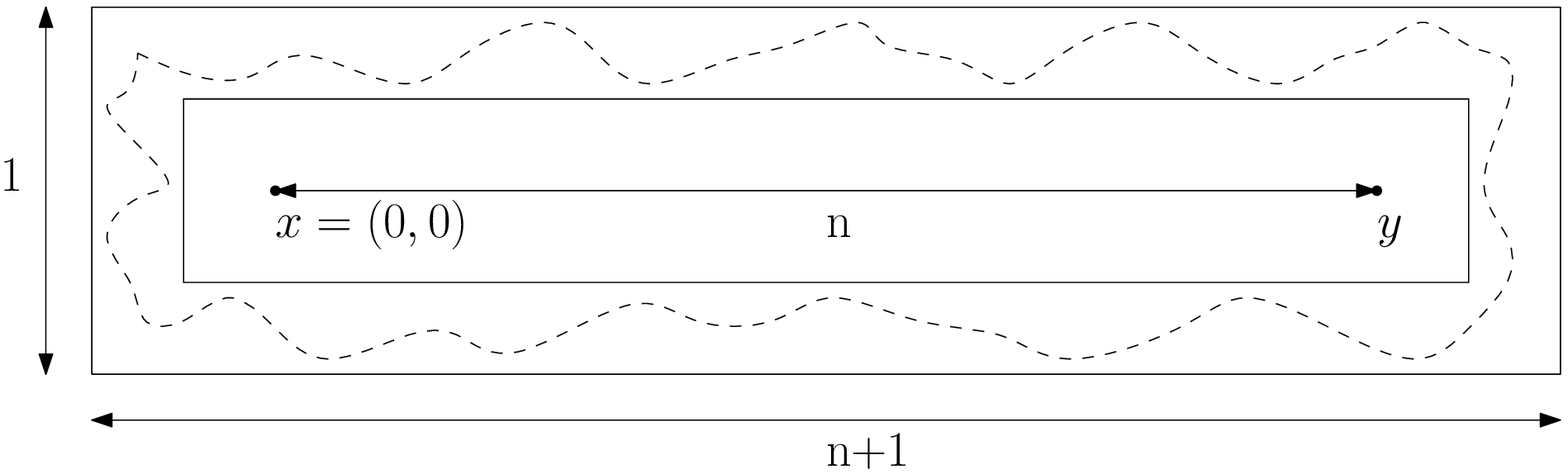}
\caption{The larger box is $R$ and the smaller one is $R_1$. The dashed curve is a blocking circuit.}\label{fig:bc}
\end{center}
\end{figure}
For ease of notation, we will assume $x=0=(0,0)$ and $y=n\vec{e}_1=(n,0)$ for some $n\in\mathbb{N}$ and $(n,0)\in a\mathbb{Z}^2$. Let $R:=[-\frac{1}{2},n+\frac{1}{2}]\times[-\frac{1}{2},\frac{1}{2}]$ and $R_1:=[-\frac{1}{4},n+\frac{1}{4}]\times[-\frac{1}{4},\frac{1}{4}]$. The dual lattice of $a\mathbb{Z}^2$ is $(a/2,a/2)+a\mathbb{Z}^2$. A dual edge is declared open (or blocking) if and only if the corresponding primal edge is closed. We refer to Section 6.1 of \cite{Gri06} for more details about duality. Let $E(R_1,R)$ be the event that there is a blocking circuit in $R\setminus R_1$ surrounding $R_1$, i.e., there is a circuit of dual open edges surrounding $R_1$ in $R^a$. See Figure \ref{fig:bc} for an illustration.
\begin{lemma}\label{lem:blo}
There exists $\epsilon_1>0$ such that for any $a\in(0,\epsilon_1]$ and $h>0$
\[\tilde{\mathbb{P}}^a_h\big(E(R_1,R)\big)\geq e^{-C_9(h)n},\]
where $C_9(h)\in (0,\infty)$ only depends on $h$.
\end{lemma}

\begin{proof}
By the self-duality of critical FK percolation with $h=0$ (see Section 6.2 of \cite{Gri06}) and RSW-type bounds (Theorem 1 of \cite{DCHN11}), there exist $\epsilon_1, c_2\in(0,1)$ such that for any $a\in(0,\epsilon_1]$
\[\mathbb{P}^a_{[0,\frac{1}{2}]\times [0,\frac{1}{4}],w,0}\left(~\exists\text{ LR dual-open crossing of }\left[0,\frac{1}{2}\right]\times \left[0,\frac{1}{4}\right]\right)\geq c_2,\]
where we abbreviate `left-right' to `LR', and `top-bottom' to `TB'.
Then Lemma \ref{lem:pos} and Remark \ref{rem:pos} imply that for any $a\in(0,\epsilon_1]$
\begin{equation}\label{eq:blo0}
\tilde{\mathbb{P}}^a_{[0,\frac{1}{2}]\times [0,\frac{1}{4}],w,h}\left(~\exists\text{ LR dual-open crossing of }\left[0,\frac{1}{2}\right]\times \left[0,\frac{1}{4}\right]\right)\geq c_2(h),
\end{equation}
where $c_2(h)\in(0,1)$ only depends on $h$. Similarly, one can show that there is some $\tilde{c}_2(h)\in(0,1)$ such that
\begin{equation}\label{eq:blo1}
\tilde{\mathbb{P}}^a_{[0,\frac{1}{2}]\times [0,\frac{1}{4}],w,h}\left(~\exists\text{ TB dual-open crossing of }\left[\frac{1}{4},\frac{1}{2}\right]\times \left[0,\frac{1}{4}\right]\right)\geq \tilde{c}_2(h).
\end{equation}
Let $F$ be the intersection of the two events in \eqref{eq:blo0} and \eqref{eq:blo1}. Then applying the FKG inequality we get
\begin{equation}\label{eq:blo}
\tilde{\mathbb{P}}^a_{[0,\frac{1}{2}]\times [0,\frac{1}{4}],w,h}(F)\geq c_2(h)\tilde{c}_2(h).
\end{equation}
Note that the wired boundary condition is the worst boundary condition for $F$ to occur. The rest of the proof follows from standard arguments in the percolation literature, i.e., by pasting different crossings defined in $F$ in rotated and/or translated versions of $[0,\frac{1}{2}]\times[0,\frac{1}{4}]$ by using the FKG inequality; see Figure \ref{fig:blocking_cons}.
\end{proof}

\begin{figure}
\begin{center}
\includegraphics[width=0.9\textwidth]{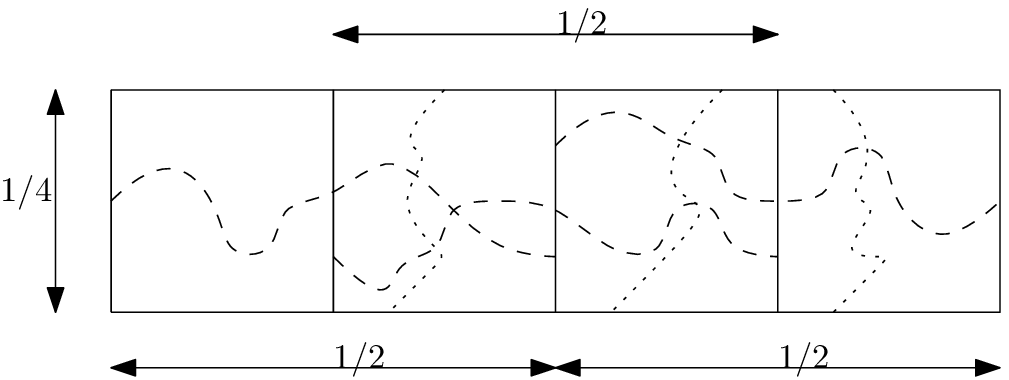}
\caption{Both dashed and dotted curves are dual-open (blocking) paths. There are $3$ overlapping rectangles of size $\frac{1}{2}\times\frac{1}{4}$. If a translated version of $F$ occurs in each of those rectangles (as shown) then there is a long horizontal dual-open crossing in the large rectangle of size $1\times \frac{1}{4}$.}\label{fig:blocking_cons}
\end{center}
\end{figure}

Next, we find a lower bound for the probability of $\{0\longleftrightarrow n\vec{e}_1\}$ under the condition that there is such a blocking circuit.
\begin{lemma}\label{lem:xy}
There exists $\epsilon_2>0$ such that for any $a\in(0,\epsilon_2]$ and $h>0$
\[\tilde{\mathbb{P}}^a_h\big(0\longleftrightarrow n\vec{e}_1\big|E(R_1,R)\big)\geq \tilde{C}_1^2 a^{1/4} e^{-C_6n},\]
where $\tilde{C}_1$ is as in Lemma \ref{lem1arm} and $C_{6}\in (0,\infty)$.
\end{lemma}

\begin{proof}
By the FKG inequality, the probability in the lemma is larger than or equal to
\[\tilde{\mathbb{P}}^a_{R_1,f,h}(0\longleftrightarrow n\vec{e}_1).\]
By the FKG inequality and RSW-type bounds (Theorem 1 of \cite{DCHN11}), there exist $\epsilon_2, c_3\in(0,1)$ such that for any $a\in(0,\epsilon_2]$
\begin{align}
&\tilde{\mathbb{P}}^a_{[0,\frac{1}{4}]\times [0,\frac{1}{2}],f,h}\left(~\exists\text{ TB open crossing of }\left[0,\frac{1}{4}\right]\times \left[0,\frac{1}{2}\right]\right)\nonumber\\
&\geq \mathbb{P}^a_{[0,\frac{1}{4}]\times [0,\frac{1}{2}],f,0}\left(~\exists\text{ TB open crossing of }\left[0,\frac{1}{4}\right]\times \left[0,\frac{1}{2}\right]\right)\nonumber\\
&\geq c_3>0.
\end{align}
Let $F_1$ be the event that there is a LR open crossing of $[0,\frac{1}{2}]\times[0,\frac{1}{2}]$ and a TB crossing of $[\frac{1}{4},\frac{1}{2}]\times [0,\frac{1}{2}]$. Then the FKG inequality implies
\begin{equation}
\tilde{\mathbb{P}}^a_{[0,\frac{1}{2}]^2,f,h}(F_1)\geq (c_3)^2.
\end{equation}
Lemma \ref{lem1arm} and the FKG inequality imply
\[\tilde{\mathbb{P}}^a_{[-\frac{1}{4},\frac{1}{4}]^2,f,h}\left(0\longleftrightarrow \partial_{in}\left(\left[-\frac{1}{4},\frac{1}{4}\right]^2\right)^a\right)\geq \tilde{C}_1 a^{1/8}.\]
By symmetry and the union bound, we have
\begin{equation}
\tilde{\mathbb{P}}^a_{[-\frac{1}{4},\frac{1}{4}]^2,f,h}\left(0\longleftrightarrow \text{the right side of }\left(\left[-\frac{1}{4},\frac{1}{4}\right]^2\right)^a\right)\geq \tilde{C}_1 a^{1/8}/4.
\end{equation}
Let $F_2$ be the event that there is  a TB open crossing of $[0,\frac{1}{4}]\times[-\frac{1}{4},\frac{1}{4}]$ and $0$ is connected to the right side of $\left([-\frac{1}{4},\frac{1}{4}]^2\right)^a$ by an open path within $\left([-\frac{1}{4},\frac{1}{4}]^2\right)^a$. Then the FKG inequality implies
\begin{equation}
\tilde{\mathbb{P}}^a_{[-\frac{1}{4},\frac{1}{4}]^2,f,h}(F_2)\geq c_3\tilde{C}_1 a^{1/8}/4.
\end{equation}
The rest of the proof follows from standard arguments in the percolation literature by considering the intersection of $F_2$, translates of $F_1$ (one for each of the overlapping squares covering $R_1$ as depicted in Figure \ref{fig:Long_path}) and the event $F_3:=\{n\vec{e}_1\longleftrightarrow \text{ the left side of its square}\}$~; see Figure \ref{fig:Long_path}.
\end{proof}

\begin{figure}
\begin{center}
\includegraphics[width=.95\textwidth]{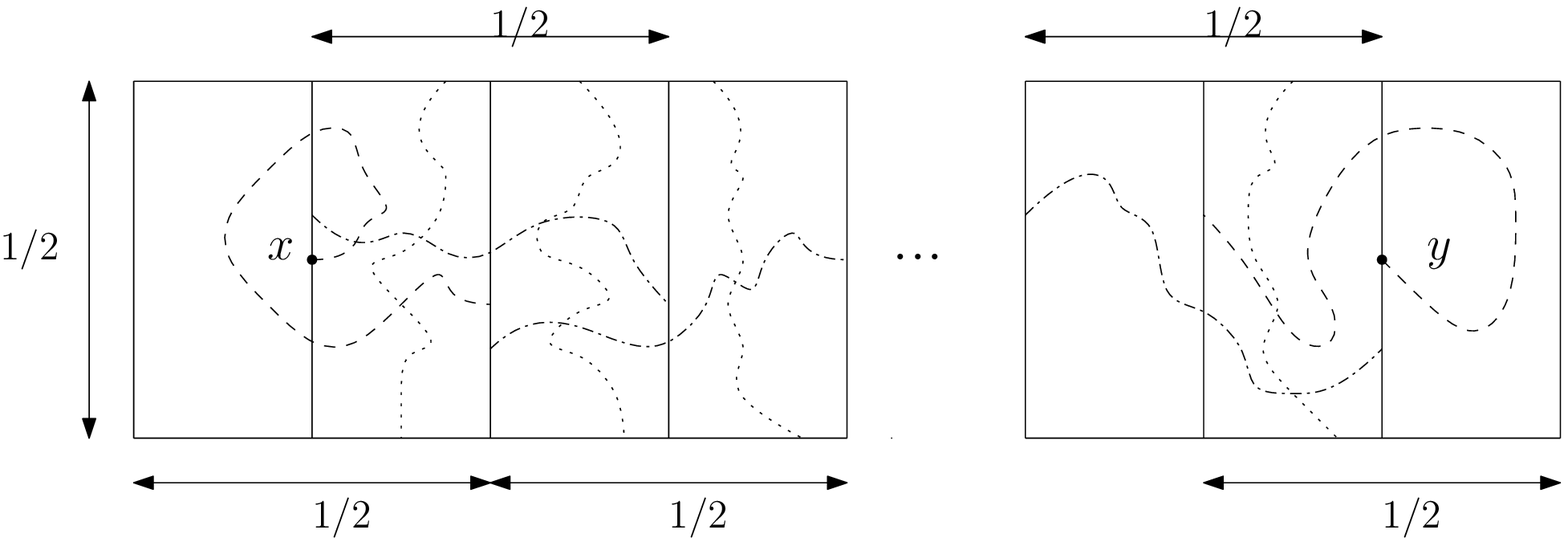}
\caption{All dashed, dotted and dash-dotted curves are open paths. There are $3$ overlapping squares of size $\frac{1}{2}\times\frac{1}{2}$ on the left and $2$ on the right. In this configuration, $F_2$ occurs in the first (leftmost) square and translated versions of $F_1$ occur in all other squares up to 
the rightmost square where $F_3$ occurs.}\label{fig:Long_path}
\end{center}
\end{figure}

Our last lemma says that, conditioned on there being a blocking circuit, the cluster of $x=0$ (denoted by $\mathcal{C}(0)$) is exponentially unlikely to be large.
\begin{lemma}\label{lem:lcd}
For each $h>0$, there exists $K(h)\in (0,\infty)$ such that for any $a\in (0,\epsilon_1]$
\begin{align}\label{eq:lcd}
\tilde{\mathbb{P}}^a_h\big(0\longleftrightarrow n\vec{e}_1,|\mathcal{C}(0)|\geq a^{-15/8}nK(h)\big|E(R_1,R)\big)
\leq \tilde{C}_1^2 a^{1/4} e^{-2C_6n}
\end{align}
where $\epsilon_1$ is as in Lemma \ref{lem:blo}, $\tilde{C}_1$ and $C_6$ are as in Lemma \ref{lem:xy}.
\end{lemma}
\begin{remark}
In Lemma \ref{lem:lcd}, $K(h)$ was chosen so that the exponential decay constant in \eqref{eq:lcd} is $2C_6$. What really matters in the proof of Theorem \ref{thm:upper} is that the rate strictly exceeds the rate $C_6$ of Lemma \ref{lem:xy}. In fact by choosing $K(h)$ large, the rate in \eqref{eq:lcd} can be made arbitrarily large.
\end{remark}
\begin{proof}
We choose $a\in(0,\epsilon_1]$ such that the probability of the conditioning event in \eqref{eq:lcd} is positive by Lemma \ref{lem:blo}.
Let $Q^a(i):=\left(Q+(i,0)\right)^a$ for $i\in\mathbb{N}$; $K(h)>0$ will be chosen later.
By the FKG inequality, the LHS of \eqref{eq:lcd} is bounded above by
\begin{align}\label{eq:lcd1}
&\tilde{\mathbb{P}}^a_{R,w,h}\big(0\longleftrightarrow n\vec{e}_1,|\mathcal{C}(0)|\geq a^{-15/8}nK(h)\big)\nonumber\\
&\leq \tilde{\mathbb{P}}^a_{R,w,h}\big(0\longleftrightarrow n\vec{e}_1,|\mathcal{C}(0)|\geq a^{-15/8}nK(h)\big|\partial_{in} Q^a(i) \text{ is open } \forall\, 0\leq i\leq n \big),
\end{align}
where $\{\partial_{in} Q^a(i)\text { is open}\}$ means that each nearest neighbor edge between two vertices in $\partial_{in} Q^a(i)$ is open. When $\partial_{in} Q^a(i)$ is open, let $A_0^i$ be the renormalized area of the boundary cluster of $Q^a(i)$. Then the RHS of \eqref{eq:lcd1} is less than or equal to (in what follows, $\{\partial Q^a\text { open}\}$ is an abbreviation of $\{\partial_{in} Q^a(i)\text { is open } \forall\, 0\leq i\leq n\}$)
\begin{align}\label{eq:lcd2}
&\tilde{\mathbb{P}}^a_{R,w,h}\big(0\longleftrightarrow \partial_{in}Q^a(0), n\vec{e}_1\longleftrightarrow \partial_{in}Q^a(n),\sum_{i=0}^n A_0^i\geq nK(h)\big| \partial Q^a \text{ open}\big)\nonumber\\
&\leq \tilde{\mathbb{P}}^a_{R,w,h}\big(0\longleftrightarrow \partial_{in}Q^a(0), n\vec{e}_1\longleftrightarrow \partial_{in}Q^a(n), A_0^0\geq n K(h)/3\big| \partial Q^a \text{ open}\big)\nonumber\\
& +\tilde{\mathbb{P}}^a_{R,w,h}\big(0\longleftrightarrow \partial_{in}Q^a(0), n\vec{e}_1\longleftrightarrow \partial_{in}Q^a(n), A_0^n\geq n K(h)/3\big| \partial Q^a \text{ open} \big)\nonumber\\
& +\tilde{\mathbb{P}}^a_{R,w,h}\big(0\longleftrightarrow \partial_{in}Q^a(0), n\vec{e}_1\longleftrightarrow \partial_{in}Q^a(n), \sum_{i=1}^{n-1} A_0^i\geq \frac{n K(h)}{3}\big| \partial Q^a \text{ open}\big)\nonumber\\
&=\Big[2\tilde{\mathbb{P}}^a_{Q,w,h}\big( A_0^0\geq n K(h)/3\big|0\longleftrightarrow \partial_{in}Q^a(0)\big)+\text{Prob}\big(\sum_{i=1}^{n-1} W^i\geq n K(h)/3\big)\Big]\nonumber\\
&\quad \times \Big[\tilde{\mathbb{P}}^a_{Q,w,h}\big(0\longleftrightarrow \partial_{in}Q^a(0)\big)\Big]^2
\end{align}
where $W_1,\ldots,W_{n-1}$ are i.i.d. random variables distributed like $A_0^0$ from $\tilde{\mathbb{P}}^a_{Q,w,h}$. Applying Propositions \ref{prop:mgfbdh}, \ref{prop:onearm} and \ref{prop:mgfcbd} and using an exponential Chebyshev inequality, we obtain that \eqref{eq:lcd2} is less than or equal to
\begin{align}\label{eq:lcd3}
&\Big[2 C_8(2h)e^{C_5\big(h+h^2+(2h)+(2h)^2\big)}e^{-nhK(h)/3}+e^{\left[\ln 2+C_5\big((2h)^2+2h\big)\right](n-1)}e^{-nhK(h)/3}\Big]\nonumber\\
&\times \Big[C_7(h)a^{1/8}\Big]^2.
\end{align}
From \eqref{eq:lcd3} one can choose $K(h)$ so large that $hK(h)/3-\left[\ln2+C_5(4h^2+2h)\right]>2C_6$ so that the lemma holds.
\end{proof}

We are ready to prove Theorem \ref{thm:upper}.
\begin{proof}[Proof of Theorem \ref{thm:upper}]
As mentioned earlier in this section, we will set $x=(0,0)$ and $y=(n,0)$. By Lemma \ref{lemES},
\begin{align*}
\langle\sigma_x;\sigma_y\rangle_{a,h}&=\mathbb{P}_h^a(x\longleftrightarrow y)-\mathbb{P}_h^a(x\longleftrightarrow g)\mathbb{P}_h^a(y\longleftrightarrow g)\\
&=\mathbb{P}_h^a(x\longleftrightarrow g, y\longleftrightarrow g)-\mathbb{P}_h^a(x\longleftrightarrow g)\mathbb{P}_h^a(y\longleftrightarrow g)\\
&\quad+\mathbb{P}_h^a(x\centernot \longleftrightarrow g, x\longleftrightarrow y).
\end{align*}
By the FKG inequality,
\[\mathbb{P}_h^a(x\longleftrightarrow g, y\longleftrightarrow g)-\mathbb{P}_h^a(x\longleftrightarrow g)\mathbb{P}_h^a(y\longleftrightarrow g)\geq0.\]
Therefore,
\begin{align}\label{eq:thmup1}
\langle\sigma_x;\sigma_y\rangle_{a,h}&\geq \mathbb{P}_h^a(x\centernot \longleftrightarrow g, x\longleftrightarrow y)\nonumber\\
&\geq \mathbb{P}_h^a\big(x\centernot \longleftrightarrow g, x\overset{a\mathbb{Z}^2}{\longleftrightarrow} y, |\mathcal{C}^{a\mathbb{Z}^2}(x)|<a^{-15/8}|x-y|K(h)\big),
\end{align}
where $\mathcal{C}^{a\mathbb{Z}^2}(x)$ is the cluster of $x$ on $a\mathbb{Z}^2$ (that is, omitting all external edges) and $K(h)$ is the same as in Lemma \ref{lem:lcd}. Let $L>0$ satisfy $(L/a)>a^{-15/8}|x-y|K(h)$. Then the event in \eqref{eq:thmup1} only depends on the status of edges in $\mathscr{B}(\Lambda_L^a)\cup\mathscr{E}(\Lambda_L^a)$.

Note that, because of the DLR property/domain Markov property, one can group boundary conditions into a finite number of sets such that two boundary conditions in the same set induce the same Gibbs measure in $\Lambda^a_L$. By summing over all such sets of boundary conditions, we see that the RHS of \eqref{eq:thmup1} is equal to
\begin{align}\label{eq:thmup2}
&\sum_{\tilde{\rho}}\mathbb{P}^a_{\Lambda_L,\tilde{\rho},h}\big(x\centernot \longleftrightarrow g, x\overset{a\mathbb{Z}^2}{\longleftrightarrow} y, |\mathcal{C}^{a\mathbb{Z}^2}(x)|<a^{-15/8}|x-y|K(h)\big)\mathbb{P}^a_h(\tilde{\rho})\nonumber\\
&=\sum_{\tilde{\rho}}\mathbb{P}^a_{\Lambda_L,\tilde{\rho},h}\big(x\centernot \longleftrightarrow g\big| x\overset{a\mathbb{Z}^2}{\longleftrightarrow} y, |\mathcal{C}^{a\mathbb{Z}^2}(x)|<a^{-15/8}|x-y|K(h)\big)\nonumber\\
&\quad\times\mathbb{P}^a_{\Lambda_L,\tilde{\rho},h}\big(x\overset{a\mathbb{Z}^2}{\longleftrightarrow} y, |\mathcal{C}^{a\mathbb{Z}^2}(x)|<a^{-15/8}|x-y|K(h)\big)\mathbb{P}^a_h(\tilde{\rho})\nonumber\\
&\geq \sum_{\tilde{\rho}}e^{-2h|x-y|K(h)}\mathbb{P}^a_{\Lambda_L,\tilde{\rho},h}\big(x\overset{a\mathbb{Z}^2}{\longleftrightarrow} y, |\mathcal{C}^{a\mathbb{Z}^2}(x)|<a^{-15/8}|x-y|K(h)\big)\mathbb{P}^a_h(\tilde{\rho})\nonumber\\
&= e^{-2h|x-y|K(h)}\times\mathbb{P}_h^a\big(x\overset{a\mathbb{Z}^2}{\longleftrightarrow} y, |\mathcal{C}^{a\mathbb{Z}^2}(x)|<a^{-15/8}|x-y|K(h)\big),
\end{align}
where the inequality follows from \eqref{eqtanh} in Proposition \ref{prop:tanh} and the elementary inequality $1-\tanh(r)\geq e^{-2r}$ when $r>0$.

For any $a\in \big(0,\min\{\epsilon_1,\epsilon_2\}\big]$ with $\epsilon_1,\epsilon_2$ given in Lemmas \ref{lem:blo} and \ref{lem:xy}, and $x,y\in\mathbb{R}^2$ with $|x-y|\geq 1$, we have by Lemmas \ref{lem:blo}, \ref{lem:xy} and \ref{lem:lcd} that
\begin{align}\label{eq:thmup3}
&\quad \mathbb{P}_h^a\big(x\overset{a\mathbb{Z}^2}{\longleftrightarrow} y, |\mathcal{C}^{a\mathbb{Z}^2}(x)|<a^{-15/8}|x-y|K(h)\big)\nonumber\\
&=\tilde{\mathbb{P}}_h^a\big(x\longleftrightarrow y, |\mathcal{C}^{a\mathbb{Z}^2}(x)|<a^{-15/8}|x-y|K(h)\big)\nonumber\\
&\geq \tilde{\mathbb{P}}_h^a\big(x\longleftrightarrow y, |\mathcal{C}^{a\mathbb{Z}^2}(x)|<a^{-15/8}|x-y|K(h), E(R_1,R)\big)\nonumber\\
&= \tilde{\mathbb{P}}_h^a\big(x\longleftrightarrow y, E(R_1,R)\big)\nonumber\\
&\quad-\tilde{\mathbb{P}}_h^a\big(x\longleftrightarrow y, |\mathcal{C}^{a\mathbb{Z}^2}(x)|\geq a^{-15/8}|x-y|K(h), E(R_1,R)\big)\nonumber\\
&=\tilde{\mathbb{P}}_h^a\big(E(R_1,R)\big)\big[\tilde{\mathbb{P}}_h^a\big(x\longleftrightarrow y\big|E(R_1,R)\big)\nonumber\\
&\quad -\tilde{\mathbb{P}}_h^a\big(x\longleftrightarrow y, |\mathcal{C}^{a\mathbb{Z}^2}(x)|\geq a^{-15/8}|x-y|K(h)\big| E(R_1,R)\big)\big]\nonumber\\
&\geq e^{-C_9(h)|x-y|}[\tilde{C}_1^2a^{1/4}e^{-C_6|x-y|}-\tilde{C}_1^2a^{1/4}e^{-2C_6|x-y|}].
\end{align}
Combining \eqref{eq:thmup1}, \eqref{eq:thmup2} and \eqref{eq:thmup3}, we have for any $a\in\big(0,\min\{\epsilon_1,\epsilon_2\}\big]$ that
\begin{equation}\label{eq:thmup40}
\langle\sigma_x;\sigma_y\rangle_{a,h}\geq C_{10}(h)a^{1/4}e^{-C_{11}(h)|x-y|} \text{ for any } x, y \in a\mathbb{Z}^2 \text{ with }|x-y|\geq 1,
\end{equation}
where $C_{10}(h), C_{11}(h)\in (0,\infty)$ only depend on $h$. Equation \eqref{eq:thmup40} implies (by rescaling the lattice spacing by $1/\min\{\epsilon_1,\epsilon_2\}$) that for any $a\in(0,1]$
\begin{equation}\label{eq:thmup4}
\langle\sigma_x;\sigma_y\rangle_{a,h}\geq C_{12}(h)a^{1/4}e^{-C_{13}(h)|x-y|} \text{ for any } x, y \in a\mathbb{Z}^2 \text{ with }|x-y|\geq C_2,
\end{equation}
where $C_{12}(h), C_{13}(h)\in (0,\infty)$ only depend on $h$ and $C_2=1/\min\{\epsilon_1,\epsilon_2\}$.

Now letting $a=H^{8/15}\in(0,1]$ and $h=1$ in \eqref{eq:thmup4}, we have
\begin{equation}\label{eq:thmup5}
\langle\sigma_x;\sigma_y\rangle_{H^{8/15},1}\geq C_{12}(1)H^{2/15}e^{-C_{13}(1)|x-y|}, \,x, y \in H^{8/15}\mathbb{Z}^2 \text{ with }|x-y|\geq C_2.
\end{equation}
Rewriting \eqref{eq:thmup5} on the $\mathbb{Z}^2$ lattice, we have (setting $x^{\prime}=xH^{-8/15}$ and $y^{\prime}=yH^{-8/15}$) that for any  $x^{\prime}, y^{\prime} \in \mathbb{Z}^2$ with $|x^{\prime}-y^{\prime}|\geq C_2H^{-8/15}$
\begin{equation}\label{eq:thmup6}
\langle\sigma_{x^{\prime}};\sigma_{y^{\prime}}\rangle_{1,H}\geq C_{12}(1)H^{2/15}e^{-C_{13}(1)H^{8/15}|x^{\prime}-y^{\prime}|}.
\end{equation}
This completes the proof of \eqref{eq:up2}. Then \eqref{eq:up1} follows by rewriting \eqref{eq:thmup6} on the $a\mathbb{Z}^2$ lattice with external field $a^{15/8}h$.
\end{proof}

\section{Proofs of Theorems \ref{thm2} and \ref{thm3}}\label{sec:proof}
In this section, we prove Theorems \ref{thm2} and \ref{thm3}. We first prove the following ancillary proposition.
\begin{proposition}\label{prop:mgfcon}
Suppose $D$ is a simply-connected and bounded domain in $\mathbb{R}^2$ with piecewise smooth boundary. Then for any $\tilde{f}\in C^{\infty}(\bar{D})$,
\begin{equation}
E^0_{D,f,0}(e^{\Phi^0_{D}(\tilde{f})})=\mathbb{E}^0_{D,f,0}\left(\prod_{C\in\mathscr{C}(D,f,\cdot)}\cosh\left(\mu^0_{\mathcal{C}}(\tilde{f})\right)\right)<\infty.
\end{equation}
\end{proposition}
\begin{proof}
For $\omega\in\{0,1\}^{\mathscr{B}(D^a)}$, let $\mathscr{C}_{\epsilon}(D^a,f,\omega)$ denote the collection of clusters of $\omega$ having diameter (using Euclidean distance) larger than or equal to $\epsilon$. I.e.,
\[\mathscr{C}_{\epsilon}(D^a,f,\omega):=\{\mathcal{C}:\mathcal{C}\in\mathscr{C}(D^a,f,\omega),\text{diam}(\mathcal{C})\geq\epsilon\}.\]
Similarly, we define $\mathscr{C}_{\epsilon}(D,f,\cdot)$ as the limit of $\mathscr{C}_{\epsilon}(D^a,f,\cdot)$ as $a\downarrow 0$ (see Theorem 2.1 of \cite{CCK17}). Theorem~8.2 of \cite{CCK17} says that
\begin{equation}\label{eq:wc}
\{\mu^a_{\mathcal{C}}:\mathcal{C}\in \mathscr{C}_{\epsilon}(D^a,f,\omega)\}\Longrightarrow \{\mu^0_{\mathcal{C}}:\mathcal{C}\in \mathscr{C}_{\epsilon}(D,f,\omega)\} \text{ as }a\downarrow 0,
\end{equation}
where $\Longrightarrow$ means convergence in distribution and the topology of convergence is defined by the metric in \eqref{eq:metric}. Note that there are finitely many elements in $\mathscr{C}_{\epsilon}(D^a,f,\omega)$ a.s. By the Edwards-Sokal coupling, we have
\begin{align}\label{eq:ES}
\mathbb{E}^a_{D,f,0}\left(\prod_{\mathcal{C}\in\mathscr{C}_{\epsilon}(D^a,f,\cdot)}\cosh\left(\mu^a_{\mathcal{C}}(\tilde{f})\right)\right)=\hat{\mathbb{E}}^a_{D,f,0}\left(\exp{\left(\sum_{\mathcal{C}\in\mathscr{C}_{\epsilon}(D^a,f,\cdot)}\sigma(\mathcal{C})\mu^a_{\mathcal{C}}(\tilde{f})\right)}\right),
\end{align}
where the $\sigma(\mathcal{C})$'s are i.i.d. symmetric $(\pm 1)$-valued random variables independent of everything else. Using the inequality $\cosh^2(r)\leq \cosh(2r)$ for any $r>0$, we have
\begin{align}\label{eq:supbd1}
\sup_{a>0}\mathbb{E}^a_{D,f,0}\left(\prod_{\mathcal{C}\in\mathscr{C}_{\epsilon}(D^a,f,\cdot)}\cosh\left(\mu^a_{\mathcal{C}}(\tilde{f})\right)\right)^2&\leq \sup_{a>0}\mathbb{E}^a_{D,f,0}\left(\prod_{\mathcal{C}\in\mathscr{C}_{\epsilon}(D^a,f,\cdot)}\cosh\left(2\mu^a_{\mathcal{C}}(\tilde{f})\right)\right)\nonumber\\
&\leq\sup_{a>0}\mathbb{E}^a_{D,f,0}\left(\prod_{\mathcal{C}\in\mathscr{C}(D^a,f,\cdot)}\cosh\left(2\mu^a_{\mathcal{C}}(\tilde{f})\right)\right)\nonumber\\
&\leq\sup_{a>0}\mathbb{E}^a_{D,f,0}\left(\prod_{\mathcal{C}\in\mathscr{C}(D^a,f,\cdot)}\cosh\left(2\|\tilde{f}\|_{\infty}\mu^a_{\mathcal{C}}(D^a)\right)\right)\nonumber\\
&=\sup_{a>0} E^a_{D,f,0}(e^{2\|\tilde{f}\|_{\infty}m^a_D})\leq C(\tilde{f},D)
\end{align}
where the last equality follows from Proposition \ref{prop:mgf}, and $C(\tilde{f},D)\in(0,\infty)$ only depends on $\tilde{f},D$, and the last inequality follows by considering a square with $+$ boundary conditions containing $D$, using the GKS inequalities and Lemma \ref{lem:mgf} (see \eqref{eq:GKS}). \eqref{eq:ES} with $\mu_{\mathcal{C}}^a(\tilde{f})$ replaced by $2\mu_{\mathcal{C}}^a(\tilde{f})$ and \eqref{eq:supbd1} imply
\begin{equation}\label{eq:supbd2}
\sup_{a>0}\hat{\mathbb{E}}^a_{D,f,0}\left(\exp{\left(\sum_{\mathcal{C}\in\mathscr{C}_{\epsilon}(D^a,f,\cdot)}\sigma(\mathcal{C})\mu^a_{\mathcal{C}}(\tilde{f})\right)}\right)^2\leq C(\tilde{f},D).
\end{equation}
Equation \eqref{eq:supbd1} implies that $\{\prod_{\mathcal{C}\in\mathscr{C}_{\epsilon}(D^a,f,\cdot)}\cosh\left(\mu^a_{\mathcal{C}}(\tilde{f})\right):a>0\}$ is uniformly integrable, so combining that with \eqref{eq:wc}, we obtain
\begin{equation}\label{eq:lit1}
\lim_{a\downarrow 0}\mathbb{E}^a_{D,f,0}\left(\prod_{\mathcal{C}\in\mathscr{C}_{\epsilon}(D^a,f,\cdot)}\cosh\left(\mu^a_{\mathcal{C}}(\tilde{f})\right)\right)
=\mathbb{E}^0_{D,f,0}\left(\prod_{\mathcal{C}\in\mathscr{C}_{\epsilon}(D,f,\cdot)}\cosh\left(\mu^0_{\mathcal{C}}(\tilde{f})\right)\right).
\end{equation}
From Lemma 3.3 in \cite{CCK17}, we know that
\begin{equation}\label{eq:wc1}
\sum_{\mathcal{C}\in\mathscr{C}_{\epsilon}(D^a,f,\cdot)}\sigma(\mathcal{C})\mu^a_{\mathcal{C}}(\tilde{f})\Longrightarrow\sum_{\mathcal{C}\in\mathscr{C}_{\epsilon}(D,f,\cdot)}\sigma(\mathcal{C})\mu^0_{\mathcal{C}}(\tilde{f}) \text{ as }a\downarrow 0.
\end{equation}
Equation \eqref{eq:supbd2} implies that $\{\exp{\left(\sum_{\mathcal{C}\in\mathscr{C}_{\epsilon}(D^a,f,\cdot)}\sigma(\mathcal{C})\mu^a_{\mathcal{C}}(\tilde{f})\right)}:a>0\}$ is uniformly integrable, so combining that with \eqref{eq:wc1}, we obtain
\begin{align}\label{eq:lit2}
\lim_{a\downarrow0}\hat{\mathbb{E}}^a_{D,f,0}\left(\exp{\left(\sum_{\mathcal{C}\in\mathscr{C}_{\epsilon}(D^a,f,\cdot)}\sigma(\mathcal{C})\mu^a_{\mathcal{C}}(\tilde{f})\right)}\right)=\hat{\mathbb{E}}^0_{D,f,0}\left(\exp{\left(\sum_{\mathcal{C}\in\mathscr{C}_{\epsilon}(D,f,\cdot)}\sigma(\mathcal{C})\mu^0_{\mathcal{C}}(\tilde{f})\right)}\right).
\end{align}
Equations \eqref{eq:ES}, \eqref{eq:lit1} and \eqref{eq:lit2} imply that
\begin{align}\label{eq:eq0}
\mathbb{E}^0_{D,f,0}\left(\prod_{\mathcal{C}\in\mathscr{C}_{\epsilon}(D,f,\cdot)}\cosh\left(\mu^0_{\mathcal{C}}(\tilde{f})\right)\right)=\hat{\mathbb{E}}^0_{D,f,0}\left(\exp{\left(\sum_{\mathcal{C}\in\mathscr{C}_{\epsilon}(D,f,\cdot)}\sigma(\mathcal{C})\mu^0_{\mathcal{C}}(\tilde{f})\right)}\right).
\end{align}
By the monotone convergence theorem, we have
\begin{align}\label{eq:lit3}
\lim_{\epsilon\downarrow0}\mathbb{E}^0_{D,f,0}\left(\prod_{\mathcal{C}\in\mathscr{C}_{\epsilon}(D,f,\cdot)}\cosh\left(\mu^0_{\mathcal{C}}(\tilde{f})\right)\right)=\mathbb{E}^0_{D,f,0}\left(\prod_{\mathcal{C}\in\mathscr{C}(D,f,\cdot)}\cosh\left(\mu^0_{\mathcal{C}}(\tilde{f})\right)\right)<\infty,
\end{align}
where the last inequality follows from \eqref{eq:supbd1}.
Theorem 3.4 of \cite{CCK17} says
\[\sum_{\mathcal{C}\in\mathscr{C}_{\epsilon}(D,f,\cdot)}\sigma(\mathcal{C})\mu^0_{\mathcal{C}}(\tilde{f})\Longrightarrow \sum_{\mathcal{C}\in\mathscr{C}(D,f,\cdot)}\sigma(\mathcal{C})\mu^0_{\mathcal{C}}(\tilde{f})\text{ as }\epsilon\downarrow 0.\]
Equation \eqref{eq:supbd2} also implies that $\{\exp{\left(\sum_{\mathcal{C}\in\mathscr{C}_{\epsilon}(D,f,\cdot)}\sigma(\mathcal{C})\mu^0_{\mathcal{C}}(\tilde{f})\right)}:\epsilon>0\}$ is uniformly integrable, and thus
\begin{align}\label{eq:lit4}
&\lim_{\epsilon\downarrow 0}\hat{\mathbb{E}}^0_{D,f,0}\left(\exp{\left(\sum_{\mathcal{C}\in\mathscr{C}_{\epsilon}(D,f,\cdot)}\sigma(\mathcal{C})\mu^0_{\mathcal{C}}(\tilde{f})\right)}\right)\nonumber\\
&=\hat{\mathbb{E}}^0_{D,f,0}\left(\exp{\left(\sum_{\mathcal{C}\in\mathscr{C}(D,f,\cdot)}\sigma(\mathcal{C})\mu^0_{\mathcal{C}}(\tilde{f})\right)}\right)=E^0_{D,f,0}(e^{\Phi^0_{D}(\tilde{f})}).
\end{align}
The proposition now follows from \eqref{eq:eq0}, \eqref{eq:lit3} and \eqref{eq:lit4}.
\end{proof}

Now we are ready to prove Theorem \ref{thm2}.
\begin{proof}[Proof of Theorem \ref{thm2}]
Since \eqref{eq:thm22} implies \eqref{eq:thm21}, we only need to prove \eqref{eq:thm22}.
The proof of Proposition 1.5 in \cite{CGN16} implies
\begin{equation}
E^0_{D,f,h}\left(e^{\Phi^h_{D}(\tilde{f})}\right)=\frac{E^0_{D,f,0}\left(e^{h\Phi^0_{D}(1_D)}e^{\Phi^0_{D}(\tilde{f})}\right)}{E^0_{D,f,0}\left(e^{h\Phi^0_{D}(1_D)}\right)}.
\end{equation}
Applying Proposition \ref{prop:mgfcon}, 
we have
\begin{equation}
E^0_{D,f,h}\left(e^{\Phi^h_{D}(\tilde{f})}\right)=\frac{\mathbb{E}^0_{D,f,0}\left(\prod_{C\in\mathscr{C}(D,f,\cdot)}\cosh\left(h\mu^0_{\mathcal{C}}(D)+\mu^0_{\mathcal{C}}(\tilde{f})\right)\right)}{\mathbb{E}^0_{D,f,0}\left(\prod_{C\in\mathscr{C}(D,f,\cdot)}\cosh\left(h\mu^0_{\mathcal{C}}(D)\right)\right)}.
\end{equation}
An elementary calculation shows that
\begin{equation}
E_{U_{\mathcal{C}}}(e^{S_{\mathcal{C},h}\mu^0_{\mathcal{C}}(\tilde{f})}) = \frac{\cosh\left(h\mu^0_{\mathcal{C}}(D)+\mu^0_{\mathcal{C}}(\tilde{f})\right)}{\cosh\left(h\mu^0_{\mathcal{C}}(D)\right)},
\end{equation}
which completes the proof.
\end{proof}

Next, we prove Theorem \ref{thm3}.
\begin{proof}[Proof of Theorem \ref{thm3}]
The proof is similar to that of Proposition \ref{prop:mgfbdh}. By \eqref{eq:RNC},
\begin{align*}
\mathbb{E}^0_{D,f,h}\left(e^{t\max_{\mathcal{C}\in\mathscr{C}(D,f,\cdot)}\mu^0_{\mathcal{C}}(D)}\right)
&=\frac{\mathbb{E}^0_{D,f,0}\left(e^{t\max_{\mathcal{C}\in\mathscr{C}(D,f,\cdot)}\mu^0_{\mathcal{C}}(D)}\prod_{\mathcal{C}\in \mathscr{C}(D,f,\cdot)}\cosh\left(h\mu^0_{\mathcal{C}}(D)\right)\right)}{\mathbb{E}_{D,f,0}^0\left[\prod_{\mathcal{C}\in \mathscr{C}(D,f,\cdot)}\cosh(h\mu^0_{\mathcal{C}}(D)\right]}\\
&\leq \mathbb{E}^0_{D,f,0}\left(e^{t\max_{\mathcal{C}\in\mathscr{C}(D,f,\cdot)}\mu^0_{\mathcal{C}}(D)}\prod_{\mathcal{C}\in \mathscr{C}(D,f,\cdot)}\cosh\left(h\mu^0_{\mathcal{C}}(D)\right)\right)\\
&\leq 2\mathbb{E}^0_{D,f,0}\left(\prod_{\mathcal{C}\in \mathscr{C}(D,f,\cdot)}\cosh\left((h+t)\mu^0_{\mathcal{C}}(D)\right)\right),
\end{align*}
where the last inequality follows from $e^{tr}\cosh(hr)\leq 2\cosh((h+t)r)$ and $\cosh(hs)\leq \cosh((h+t)s)$, valid for any $r,s\geq 0$.

Combining this with Proposition \ref{prop:mgfcon}, we have
\begin{align}
\mathbb{E}^0_{D,f,h}\left(\exp{\left(t\max_{\mathcal{C}\in\mathscr{C}(D,f,\cdot)}\mu^0_{\mathcal{C}}(D)\right)}\right)
&\leq 2\mathbb{E}^0_{D,f,0}\left(\prod_{C\in\mathscr{C}(D,f,\cdot)}\cosh\left((t+h)\mu^0_{\mathcal{C}}(D)\right)\right)\nonumber\\
&=2E^0_{D,f,0}\left(e^{(t+h)\Phi^0_{D}(1_D)}\right)
\end{align}
The proof is completed by using Proposition 2.2 and Theorem 1.2 of \cite{CGN16}.
\end{proof}

\section*{Appendix}
\renewcommand*{\theproposition}{\Alph{proposition}}
\setcounter{proposition}{0}
\textbf{FK-Ising coupling in a magnetic field.}
Although the material in this appendix is essentially contained in the paper by Cioletti and Vila \cite{CV16}, some of the formulas are only implicit there. We include it here for the sake of completeness.
Consider an Ising model on a finite graph $\mathcal{G}=(\mathcal{V},\mathcal{E})$ with pair ferromagnetic interactions $J_e\geq 0$ for $e\in\mathcal{E}$ and non-negative magnetic field strength $\vec{H}=(H_v:v\in\mathcal{V})$ with each $H_v\geq0$. The Gibbs measure is
\[\frac{1}{Z_{\mathcal{G}}}\exp{\left(\sum_{e=\{u,v\}}J_e\sigma_u\sigma_v+\sum_{v\in \mathcal{V}}H_v\sigma_v\right)}.\]

The Edwards-Sokal coupling in this case is a measure $\hat{\mathbb{P}}_{\vec{H}}$ for FK bond configurations and spin configurations on the extended graph $\hat{\mathcal{G}}=(\hat{\mathcal{V}},\hat{\mathcal{E}})$ where $\hat{\mathcal{V}}=\mathcal{V}\cup \{g\}$ and $\hat{\mathcal{E}}=\mathcal{E}\cup\{\{v,g\}:v\in \mathcal{V}\}$ where the cluster containing $g$ is forced to have all $\sigma_v=+1$ and all other clusters are equally likely to be $+1$ or $-1$. Below we describe a different coupling which first determines the clusters formed by only the edges in $\mathcal{E}$ and after that determines whether those clusters are connected to $g$.

Let $\tilde{\mathbb{P}}_{\mathcal{G},\vec{H}}$ (resp., $\tilde{\mathbb{P}}_{\mathcal{G},0}$ when $\vec{H}\equiv 0$) denote the FK distribution restricted to the edges in $\mathcal{G}$. For each cluster $\mathcal{C}$ in any configuration from $\tilde{\mathbb{P}}_{\mathcal{G},\vec{H}}$, the un-normalized FK measure contains a factor of $2\prod_{v\in \mathcal{C}}e^{-2 H_v}$ if \emph{none} of the $\{v,g\}$ edges to $g$ from $\mathcal{C}$ are open (the factor $2$ is because the number of clusters in $\hat{\mathcal{G}}$ is one higher than when $\mathcal{C}$ has some open edge to $g$). The sum of all remaining factors is $(1-\prod_{v\in\mathcal{C}}e^{-2H_v})$. Thus for each $\mathcal{C}$, the overall factor is
\begin{equation*}
(1-\prod_{v\in\mathcal{C}}e^{-2H_v})+2\prod_{v\in \mathcal{C}}e^{-2H_v}=1+e^{-\sum_{v\in\mathcal{C}}(2H_v)}=2e^{-\sum_{v\in\mathcal{C}}H_v}\cosh(H(\mathcal{C})),
\end{equation*}
where $H(\mathcal{C})=\sum_{v\in\mathcal{C}}H_v$. Taking the product over all clusters $\mathcal{C}$ and noting that $\sum_{\mathcal{C}}H(\mathcal{C})=\sum_{v\in\mathcal{V}}H_v$ does not depend on the FK configuration, one immediately has:
\begin{proposition}
\begin{equation}\label{eq:RNG}
\frac{d\tilde{\mathbb{P}}_{\mathcal{G},\vec{H}}}{d\tilde{\mathbb{P}}_{\mathcal{G},0}}=\frac{\prod_{\mathcal{C}}\cosh(H(\mathcal{C}))}{\tilde{\mathbb{E}}_{\mathcal{G},0}\left(\prod_{\mathcal{C}}\cosh\left(H(\mathcal{C})\right)\right)},
\end{equation}
where $\tilde{\mathbb{E}}_{\mathcal{G},0}$ is the expectation with respect to $\tilde{\mathbb{P}}_{\mathcal{G},0}$.
\end{proposition}

\begin{proposition}
Conditioned on a configuration $\tilde{\omega}$ from $\tilde{\mathbb{P}}_{\mathcal{G},\vec{H}}$, the events of whether the different clusters $\mathcal{C}_i$ in $\omega$ are connected directly to $g$ and whether the spin values, $\sigma(\mathcal{C}_i)$, are $+1$ or $-1$ are mutually independent as $i$ varies with
\begin{align}\label{eq:tanhG}
&\hat{\mathbb{P}}_{\mathcal{G},\vec{H}}(\mathcal{C}_i\longleftrightarrow g|\omega)=\tanh(H(\mathcal{C}_i))\nonumber\\
&\hat{\mathbb{P}}_{\mathcal{G},\vec{H}}(\sigma(\mathcal{C}_i)=+1|\omega)=\tanh(H(\mathcal{C}_i))+\frac{1}{2}\left(1-\tanh(H(\mathcal{C}_i))\right)\nonumber\\
&\hat{\mathbb{P}}_{\mathcal{G},\vec{H}}(\sigma(\mathcal{C}_i)=-1|\omega)=\frac{1}{2}\left(1-\tanh\left(H(\mathcal{C}_i)\right)\right).
\end{align}
\end{proposition}
\begin{proof}
This follows from the Edwards-Sokal coupling like in the proof of Proposition \ref{prop:tanh} of Section \ref{sec:coupling} above.
\end{proof}
\begin{remark*}
The analysis above extends to the FK model with cluster weight $q>0$ (see, e.g., \cite{Gri06}) where the factor in the FK measure of $2^{(\text{no. of clusters})}$ (as in \eqref{eqFKdef}) is replaced by $q^{(\text{no. of clusters})}$. This leads to a modified Radon-Nikodym factor, compared to \eqref{eq:RNG}, proportional to
\[\prod_{\mathcal{C}}\left[\frac{2}{q}\cosh(H(\mathcal{C}))+\frac{q-2}{q}e^{-H(\mathcal{C})}\right]\]
and with the RHS of the first equation in \eqref{eq:tanhG} modified to
\[\frac{\tanh\left(H(\mathcal{C})\right)}{1+(q-2)/(e^{2H(\mathcal{C})}+1)}.\]
When $q=3,4,\ldots$, and $g$ is fixed as one of the $q$ colors of the corresponding $q$-state Potts model, modified versions of the other equations in \eqref{eq:tanhG} can be easily determined.
\end{remark*}

\section*{Acknowledgements}
The research was supported in part by STCSM grant 17YF1413300 to JJ and U.S. NSF grant DMS-1507019 to CMN. The authors thank Rob van den Berg, Francesco Caravenna, Gesualdo Delfino, Roberto Fernandez, Alberto Gandolfi, Christophe Garban, Barry McCoy, Tom Spencer, Rongfeng Sun and Nikos Zygouras for useful comments and discussions related to this work. The authors benefitted from the hospitality of several units of NYU during their work on this paper: the Courant Institute and CCPP at NYU-New York, NYU-Abu Dhabi, and NYU-Shanghai. We also thank the referee for valuable comments and suggestions.


\begin{thebibliography}{99}
\bibitem{Ale98}
\textsc{K. Alexander}
(1998). On weak mixing in lattice models. \textit{Probab. Theory Relat. Fields} \textbf{110} 441-471.



\bibitem{CCK17}
\textsc{F. Camia}, \textsc{R. Conijn} and \textsc{D. Kiss}
(2017). Conformal measure ensembles for percolation and the FK-Ising model. arXiv:1507.01371v3

\bibitem{CGN14}
\textsc{F. Camia}, \textsc{C. Garban} and \textsc{C.M. Newman}
(2014). The Ising magnetization exponent on $\mathbb{Z}^2$ is $1/15$. \textit{Probab. Theory Relat. Fields} \textbf{160} 175-187.

\bibitem{CGN15}
\textsc{F. Camia}, \textsc{C. Garban} and \textsc{C.M. Newman}
(2015). Planar Ising magnetization field \upperRomannumeral{1}. Uniqueness of the critical scaling limits. \textit{Ann. Probab.} \textbf{43} 528-571.

\bibitem{CGN16}
\textsc{F. Camia}, \textsc{C. Garban} and \textsc{C.M. Newman}
(2016). Planar Ising magnetization field \upperRomannumeral{2}. Properties of the critical and near-critical scaling limits. \textit{Ann. Inst. H. Poincar\'e Probab. Statist.} \textbf{52} 146-161.

\bibitem{CJN17}
\textsc{F. Camia}, \textsc{J. Jiang} and \textsc{C.M. Newman}
(2017). Exponential decay for the near-critical scaling limit of the planar Ising model. arXiv:1707.02668v3

\bibitem{CV16}
\textsc{L. Cioletti} and \textsc{R. Vila}
(2016). Graphical representations for Ising and Potts models in general external fields. \textit{J. Stat. Phys.} \textbf{162} 81-122.

\bibitem{DCHN11}
\textsc{H. Duminil-Copin}, \textsc{C. Hongler} and \textsc{P. Nolin}
(2011). Connection probabilities and RSW-type bounds for the two-dimensional FK Ising model. \textit{Commun. Pure Appl. Math.} \textbf{64} 1165-1198.

\bibitem{ES88}
\textsc{R.G. Edwards} and \textsc{A.S. Sokal}
(1988). Generalization of the Fortuin-Kasteleyn-Swendsen-Wang representation and Monte Carlo algorithm. \textit{Phys. Rev. D.} \textbf{38} 2009-2012.



\bibitem{Gri67a}
\textsc{R.B. Griffiths}
(1967). Correlations in Ising ferromagnets. \upperRomannumeral{1}. \textit{J. Math. Phys.} \textbf{8} 478-483.


\bibitem{Gri67}
\textsc{R.B. Griffiths}
(1967). Correlations in Ising ferromagnets. \upperRomannumeral{2}. External magnetic fields. \textit{J. Math. Phys.} \textbf{8} 484-489.



\bibitem{Gri06}
\textsc{G. Grimmett}
(2006). \textit{The Random-Cluster Model}. Vol. 333, Grundlehren der Mathematischen Wissenschaften. Springer, Berlin.


\bibitem{KS68}
\textsc{D.G. Kelly}, \textsc{S. Sherman}
(1968). General Griffiths' inequalities on correlations in Ising ferromagnets. \textit{J. Math. Phys.} \textbf{9} 466-484.


\bibitem{Kis14}
\textsc{D. Kiss}
(2014). Large deviation bounds for the volume of the largest cluster in 2D critical percolation. \textit{Electron. Commun. Probab.} \textbf{19} 1-11.

\bibitem{Leb72}
\textsc{J. Lebowitz}
(1972). On the uniqueness of the equilibrium state for Ising spin systems. \textit{Commun. Math. Phys.} \textbf{25} 276-282.



\bibitem{Rue72}
\textsc{D. Ruelle}
(1972). On the use of ``small external fields'' in the problem of symmetry breakdown in statistical mechanics. \textit{Ann. Phys.} \textbf{69} 364-374.





\end{thebibliography}
\end{document}